\documentclass[11pt]{article}

\usepackage{amstext}
\usepackage{amsmath}
\usepackage{amssymb}
\usepackage{amsthm}
\usepackage{a4wide}
\usepackage{enumerate}

\usepackage{pictexwd,dcpic}

\newtheorem{theo}{Theorem}[section]
\newtheorem{lemma}[theo]{Lemma}

\newtheorem{definition}[theo]{Definition}
\newtheorem{prop}[theo]{Proposition}

\theoremstyle{definition}
\newtheorem{remark}[theo]{Remark}
\newtheorem*{example}{Example}

\newcommand{\N}{\mathbb{N}}

\newcommand{\R}{\mathbb{R}}

\newcommand{\eps}{\varepsilon}

\newcommand{\weakly}{\rightharpoonup}
\newcommand{\weakstar}{\stackrel{\ast}{\rightharpoonup}}

\newcommand{\dist}{{\rm dist}}

\def\hom{\mathop{\rm hom}}

\def\qc{\mathop{\rm qc}}
\def\sm{\mathop{\rm sym}}
\def\skw{\mathop{\rm skw}}
\def\lsc{\mathop{\rm lsc}}
\def\swlsc{\mathop{\rm swlsc}}

\def\i {\infty}

\def\sm{\mathop{\rm sym}}
\newcommand{\M}[2]{\R^{#1 \times #2}}

\def\Xint#1{\mathchoice
   {\XXint\displaystyle\textstyle{#1}}%
   {\XXint\textstyle\scriptstyle{#1}}%
   {\XXint\scriptstyle\scriptscriptstyle{#1}}%
   {\XXint\scriptscriptstyle\scriptscriptstyle{#1}}%
   \!\int}
\def\XXint#1#2#3{{\setbox0=\hbox{$#1{#2#3}{\int}$}
     \vcenter{\hbox{$#2#3$}}\kern-.5\wd0}}

\def\dashint{\Xint-}

\usepackage{color}

\begin{document}

\begin{center}
\begin{Large}
Closure and commutability results for Gamma-limits and the geometric linearization and homogenization of multi-well energy functionals
\end{Large}
\end{center}

\begin{center}
\begin{large}
Martin Jesenko and Bernd Schmidt\\[0.2cm]
\end{large}
\begin{small}
Institut f{\"u}r Mathematik,\\ 
Universit{\"a}t Augsburg\\ 
86135 Augsburg, Germany\\ 
{\tt martin.jesenko@math.uni-augsburg.de}\\
{\tt bernd.schmidt@math.uni-augsburg.de}
\end{small}
\end{center}

\begin{abstract}
Under a suitable notion of equivalence of integral densities we prove a $\Gamma$-closure theorem for integral functionals: The limit of a sequence of $\Gamma$-convergent families of such functionals is again a $\Gamma$-convergent family. Its $\Gamma$-limit is the limit of the $\Gamma$-limits of the original problems. This result not only provides a common basic principle for a number of linearization and homogenization results in elasticity theory. It also allows for new applications as we exemplify by proving that geometric linearization and homogenization of multi-well energy functionals commute. 
\end{abstract}

\tableofcontents

\section{Introduction}

Over the last decades there has been an ever growing interest in devising effective theories for complex systems in the natural sciences and engineering. In many situations it is a major mathematical challenge to derive such a reduced theory in a rigorous way as a limiting theory of basic physical principles. In this paper we are interested in variational models and in particular in applications to elasticity theory. A classical and elementary example is the derivation of linearized elasticity from non-linear elasticity in the limit of small displacements through a simple Taylor expansion of the stored energy function. Yet, the mathematically rigorous derivation, which moreover guarantees that energy minimizers of the non-linear theory converge to energy minimizers in the linearized regime, is non-trivial and has only been solved comparatively recently by Dal Maso, Negri and Percivale \cite{DalMasoNegriPercivale}. With a view to the applications of our main abstract result, we also mention here the problem of finding geometrically linearized theories for energy functionals with multiple wells as studied by the second author in \cite{Schmidt}. An example of a complex multiscale system is a composite material with highly oscillatory material properties. Here an effective theory is provided by a homogenization procedure of the stored energy function of the material, rigorously derived by Braides \cite{Braides:85} and M{\"u}ller \cite{Mueller}. We will discuss such materials in detail below and in particular revisit the question if homogenization and linearization commute which has been addressed by M{\"u}ller and Neukamm \cite{MuellerNeukamm} and then by Gloria and Neukamm \cite{GloriaNeukamm}. 

The main aim of this paper is to prove an abstract `$\Gamma$-closure result' for sequences of $\Gamma$-convergent families: We show that under a suitable notion of equivalence of integral densities the limit of a $\Gamma$-convergent family of functionals is again a $\Gamma$-convergent family and the $\Gamma$-limit is given as a limit of the sequence of $\Gamma$-limits of the original problems. In this sense we in fact prove a stability result for $\Gamma$-convergent families. Our motivation comes from a homogenization closure theorem proved by Braides in \cite{Braides:86} (see also \cite{BraidesDefranceschi}), which enabled him to extend his homogenizability results to a class of almost periodic functions. Also the aforementioned theorem of Gloria and Neukamm \cite{GloriaNeukamm} is in fact a closure theorem in the special case of elastic energy functionals with quadratic expansion at the identity for small displacements.

The interest in such a $\Gamma$-closure theorem is twofold. Firstly, it sheds a new light to a number of problems that have been studied over the last years. In particular, we will see that the results on linearization in \cite{DalMasoNegriPercivale}, on geometric linearization in \cite{Schmidt} and on commutability in \cite{MuellerNeukamm,GloriaNeukamm} are direct consequences of our abstract theorem. Secondly, our general method also allows for new applications. As an example we focus here on composite materials with multiple energy wells in the limit of highly oscillatory mixtures and small displacements and show that homogenization and geometric linearization commute. However, while all these examples concern the behavior of elastic materials at small displacements, our general scheme is not restricted to applications in elasticity theory. In contrast, we presume that our results are of interest in a variety of different problems. Our $\Gamma$-closure theorems allow to analyze the effective properties of multiscale problems with two limiting parameters: They provide general criteria that guarantee commutability of the limits and stability results for simultaneous limits of the parameters. We believe that such results may be of particular interest, e.g., in numerical schemes for multiscale problems, where a thorough understanding of the interplay of the small model parameter and the length scale of the numerical discretization is crucial. 

The paper is organized as follows. In Section \ref{section:Gamma-closure} we prove our main abstract $\Gamma$-closure results for integral functionals first under standard growth assumptions on the densities in Theorems \ref{theo:Gamma-closure-single} and \ref{theo:Gamma-closure-variable}. These theorems generalize a homogenization closure result of Braides for quasiconvex integrands with standard growth, cf.\ \cite{Braides:86,BraidesDefranceschi}, to general integrands inducing $\Gamma$-convergent families. A key ingredient of our proof is an equiintegrability theorem by Fonseca, M{\"u}ller and Pedregal \cite{FonsecaMuellerPedregal}. Next we extend these results in our main Theorems \ref{theo:Gamma-closure-single-Garding} and \ref{theo:Gamma-closure-variable-Garding} to functionals satisfying a `$p$-G{\r a}rding type inequality' (see Definition \ref{definition:Garding}). This generalization is motivated by our applications to elasticity theory, where realistic models in the small strain limit are incompatible with the assumption of standard $p$-growth assumptions from below. In Theorem \ref{theo:Gamma-closure-boundary-values} we also consider these functionals with pre-assigned boundary values. 

In Section \ref{section:General-applications} we collect a number of immediate consequences of our $\Gamma$-closure theorems. Specializing to one parameter families, in Theorems \ref{theo:perturbation} and \ref{theo:relaxation} we first obtain a perturbation result for $\Gamma$-convergent functionals and a relaxation result for sequences of integral functionals. We then note in Theorems \ref{theo:commutability} and \ref{theo:simultaneous-limits} that under a suitable equivalence assumption on the densities the $\Gamma$-limits of a doubly indexed family of functionals commute and that in fact every diagonal sequence produces the same $\Gamma$-limit. Our findings are finally specialized to the problem of homogenizing integral functionals of G{\r a}rding type in Theorems \ref{theo:homogenization-closure} and \ref{theo:homogenization-commutability}. 

The last Section \ref{section:Elasticity} is devoted to applications in elasticity and in particular to the simultaneous homogenization and geometric linearization of multiwell energies. The geometric rigidity result of Friesecke, James and M{\"u}ller in \cite{FrieseckeJamesMueller} implies that stored energy functions in realistic models induce G{\r a}rding type energy functionals to which our theory developed in the previous sections applies. We thus see that indeed homogenization and geometric linearization commute. We formulate Theorem \ref{theo:multiwells} similarly as in \cite{Braides:86,GloriaNeukamm} in terms of general `homogenizable' densities and note that this includes, in particular, the case of periodic and the case of ergodic stochastic material mixtures. Moreover, we remark that all $\Gamma$-convergence statements can be complemented by observing compactness for finite energy sequences with pre-assigned boundary values, see Remark \ref{remark:muwe},\ref{rem:boundary-values-compactness}. Finally, we also note here that the (geometric) linearization and commutability results in \cite{DalMasoNegriPercivale,Schmidt,MuellerNeukamm,GloriaNeukamm} are a direct consequence of Theorem \ref{theo:multiwells}.

\section{A $\Gamma$-closure result}\label{section:Gamma-closure}

Let $\Omega$ be an open subset of $\R^{n}$. For a (doubly indexed) Borel function $f^{(j)}_{\eps} : \Omega \times \M{m}{n} \to \R$ bounded from below and $U \subset \Omega$ bounded and open we define the integral functional $F^{(j)}_{\eps}$ by 
\begin{align*}\label{eq:F-defi}
  F^{(j)}_{\eps}(u, U) 
  := \int_{ U } f^{(j)}_{\eps} (x, \nabla u(x)) \, dx 
\end{align*}
for $u \in W^{1, p}(\Omega; \R^m)$ and taking the value $+ \infty$ otherwise. If $\Omega$ itself is bounded we simply write $F^{(j)}_{\eps}(\cdot, \Omega) = F^{(j)}_{\eps}$. In view of our applications and also for notational clarity we choose the indices $j \in \N \cup \{\infty\}$ and $\eps$ as the elements of a positive null sequence or $0$. 

Our main aim is to provide a rather general set of conditions which in applications are easy to check that allow for a $\Gamma$-closure result of the following type: If the functionals $F^{(j)}_{\eps}$ $\Gamma$-converge as $\eps \to 0$ for each $j \in \N$ and the densities $f^{(j)}_{\eps}$ are close to $f^{(\infty)}_{\eps}$ for $\eps > 0$ and large $j$ in a suitable sense, then also the $\Gamma$-limit for $j = \infty$ exists and is given as a limit as $j \to \infty$ of the $\Gamma$-limits for finite $j$. 

\subsection{$\Gamma$-closure under standard growth assumptions}

As a first step, in this section we consider densities of standard $p$-growth. More precisely, we state the following 
\begin{definition}
\begin{itemize}
\item[(i)] We say that the families $((f^{(j)}_{\eps})_{\eps > 0})_{j \in \N}$ and $(f^{(\infty)}_{\eps})_{\eps > 0}$ are equivalent on $U \subset \Omega$ open, if  
\[ \lim_{j \to \infty} \limsup_{\eps \to 0} \int_{ U } 
   \sup_{|X| \le R} | f^{(j)}_{\eps}(x, X) - f^{(\infty)}_{\eps}(x , X) | \, dx = 0 \]
for every $R \ge 0$.

\item[(ii)] Recall that a family $f^{(j)}_{\eps} : \Omega \times \M{m}{n} \to \R$ of functions is said to uniformly fulfil a standard $p$-growth condition, $1 < p < \infty$, if there are $ \alpha, \beta > 0 $ independent of $j$ and $\eps$ such that
\begin{align*}
  \alpha |X|^{p} - \beta \le f^{(j)}_{\eps}(x,X) \le \beta( |X|^{p} + 1 ) 
\end{align*}
for almost all $x \in \Omega $ and all $X \in \M{m}{n}$. 
\end{itemize}
\end{definition}

\begin{theo}[$\Gamma$-closure on a single domain]\label{theo:Gamma-closure-single}
Let $\Omega \subset \R^n$ be bounded and open. Suppose that the family of Borel functions $f^{(j)}_{\eps} : \Omega \times \M{m}{n} \to \R$, $ j \in \N \cup \{ \infty \}$, $ \eps > 0$, uniformly fulfils a standard $p$-growth condition. Assume that
\begin{itemize} 
\item[(i)] for each $j < \infty$ the $\Gamma$-limit $\Gamma(L^p)\text{-}\lim_{\eps \to 0} F^{(j)}_{\eps} =: F^{(j)}_0$ exists and  

\item[(ii)] the families $((f^{(j)}_{\eps})_{\eps > 0})_{j \in \N}$ and $(f^{(\infty)}_{\eps})_{\eps > 0}$ are equivalent on $\Omega$. 
\end{itemize}

Then also $\Gamma(L^p)\text{-}\lim_{\eps \to 0} F^{(\infty)}_{\eps} =: F^{(\infty)}_0$ exists and is the pointwise and the $\Gamma$-limit of $F^{(j)}_0$ as $j \to \infty$: 
\[ F^{(\infty)}_0 
   = \lim_{j \to \infty} F^{(j)}_0 
   = \Gamma(L^p)\text{-}\lim_{j \to \infty} F^{(j)}_0. \] 
\end{theo}

\begin{remark} 
\begin{enumerate} 
\item Note that by Theorem \ref{theo:standard-growth-Gamma-compactness} also the functionals $F^{(j)}_0$, $j \in \N \cup \{\infty\}$, are integral functionals with Carath\'{e}odory densities $f^{(j)}_0$ of standard $p$-growth. In fact, being densities of $\Gamma$-limits, the $f^{(j)}_0$ are quasiconvex in the second argument. 

\item Our definition of equivalence of the sequences $((f^{(j)}_{\eps})_{\eps > 0})_{j \in \N}$ and $(f^{(\infty)}_{\eps})_{\eps > 0}$ is motivated by the notion of equivalence in the homogenization closure theorem \ref{theo:homogenization-closure}, which has been introduced in \cite{Braides:86}. 
\end{enumerate}
\end{remark}

The previous result can be extended to functionals on variable domains in a straightforward manner: For $\Omega \subset \R^n$ open (not necessarily bounded), denote by ${\cal A}(\Omega)$ the set of bounded open subsets of $\Omega$ with Lipschitz boundary. 

\begin{theo}[$\Gamma$-closure on variable domains]\label{theo:Gamma-closure-variable}
Let $\Omega \subset \R^n$ be open. Suppose that the family of Borel functions $f^{(j)}_{\eps} : \Omega \times \M{m}{n} \to \R$, $ j \in \N \cup \{ \infty \}$, $ \eps > 0$, uniformly fulfils a standard $p$-growth condition. Assume that
\begin{itemize} 
\item[(i)] For each $j < \infty$ and $U \in {\cal A}(\Omega)$ the $\Gamma$-limit $\Gamma(L^p)\text{-}\lim_{\eps \to 0} F^{(j)}_{\eps}(\cdot, U)$ exists. 

\item[(ii)] The families $((f^{(j)}_{\eps})_{\eps > 0})_{j \in \N}$ and $(f^{(\infty)}_{\eps})_{\eps > 0}$ are equivalent on every $U \in {\cal A}(\Omega)$. 
\end{itemize}

Then for each $ j \in \N \cup \{ \i \} $ there exists a Carath{\'e}odory function $ f^{(j)}_{0} : \Omega \times \M{m}{n} \to \R $, uniquely determined a.e.\ on $\Omega$, 
such that for the corresponding integral functional $F^{(j)}_{0}$ it holds 
\begin{align*}
  \Gamma(L^p)\text{-}\lim_{\eps \to 0} F^{(j)}_{\eps}(\cdot, U)
   = F^{(j)}_0(\cdot, U)
\end{align*}
for all $U \in {\cal A}(\Omega)$. Moreover, for every $ u \in L^{p}( \Omega ; \R^{m} ) $
\[ F^{(\infty)}_0(u, U) 
   = \lim_{j \to \infty} F^{(j)}_0(u, U) 
   = \Gamma(L^p)\text{-}\lim_{j \to \infty} F^{(j)}_0(u, U) \] 
and the limiting densities $ f^{(j)}_{0} $ for all $X \in \M{m}{n}$ satisfy
\[ f^{(j)}_0(\cdot, X) \weakstar f^{(\infty)}_0 (\cdot, X) \quad \mbox{in } L^{\infty}(\Omega). \] 
\end{theo}

\begin{remark}
It is worth noting that for any family $((f^{(j)}_{\eps})_{ \eps>0 })_{j \in \N}$ there is always a subsequence $\eps_k$ of $\eps$ such that $\Gamma(L^p)\text{-}\lim_{ \eps \to 0 } F^{(j)}_{\eps}(\cdot, U)$ exists for all $j \in \N$ and $U \in {\cal A}( \Omega )$. This follows from Theorem \ref{theo:standard-growth-Gamma-compactness} in combination with a standard diagonal sequence argument. 
\end{remark}
\begin{figure}[h!]
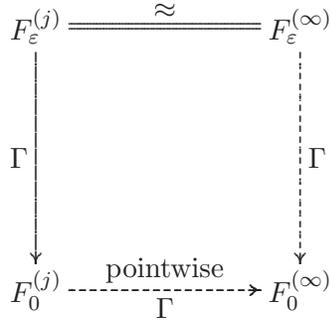

\null\hspace{5.3cm}
\begindc{\commdiag}[500]
\obj(0,2)[aa]{$ F^{(j)}_{ \eps } $}
\obj(2,2)[bb]{$ F^{(\infty)}_{ \eps } $}
\obj(0,0)[cc]{$ F^{(j)}_{0} $}
\obj(2,0)[dd]{$ F^{(\infty)}_{0} $}
\mor{aa}{bb}{$ \approx $}[\atleft, \equalline ]
\mor{bb}{dd}{$ \Gamma $}[\atleft,1]
\mor{aa}{cc}{$ \Gamma $}[\atright,0]
\mor{cc}{dd}{$ \Gamma $}[\atright,1]
\mor{cc}{dd}{pointwise}[\atleft,1]
\enddc
\caption{\label{fig:Gamma-closure}Schematically, our theorems can be summarized by the above diagram. The assumptions are given by solid lines, where in particular the equivalence of densities is indicated by a double line. The consequences of the theorems are given by dashed arrows. }
\end{figure}
\begin{proof}[Proof of Theorem \ref{theo:Gamma-closure-single}]
The proof is divided into three steps. First we assume that 
\[ \Gamma(L^p)\text{-}\lim_{\eps \to 0} F^{(\infty)}_{\eps} 
   =: F^{(\i)}_0 \]
exists and that $ F^{(\i)}_{0}(u) < \i $ if and only if $ u \in W^{1, p}( U ; \R^m ) $. This will be justified in Step 3.
\item
Step 1: Upper bound. For $u \in L^p(\Omega; \R^m)$ we claim that 
\begin{align}\label{eq:G-upperbound}
  \limsup_{j \to \infty} F^{(j)}_{0}(u) 
  \le F^{(\i)}_{0}(u). 
\end{align}
\item
This is obvious if $u \in L^p( \Omega ; \R^m) \setminus W^{1, p}( \Omega ; \R^m)$. 
If $ u \in W^{1,p}( \Omega ; \R^m)$ we choose a recovery sequence $ (u_{\eps})_{\eps} $ for $u$ with $F^{(\infty)}_{\eps}(u_{\eps}) \to F^{(\i)}_{0}(u)$. 
By equicoercivity it follows that $ ( u_{\eps} )_{\eps} $ is bounded in $W^{1, p}(\Omega; \R^m)$. Hence, by Lemma \ref{lemma:equi-int} there exists a subsequence $ ( u_{\eps_k} )_{ k \in \N } $ and functions $ v_{k} \in W^{1, p}(\Omega; \R^m) $ 
such that $ ( |\nabla v_{k}|^p )_{k \in \N} $ is equiintegrable on $\Omega$ and
\begin{align*} 
  |A_k| 
  \to 0 
  \mbox{ for } 
  A_k := \{ x \in \Omega : \nabla u_{\eps_k}(x) \ne \nabla v_{k}(x) \} \quad \mbox{and} \quad
  v_{k} 
  \weakly u \mbox{ in } W^{1,p}(\Omega; \R^m) 
\end{align*}
as $k \to \infty$. Setting $E^M_k := \{x \in \Omega : |\nabla v_{k}(x)| \ge M \}$ we obtain that 
\begin{align*}
  F^{(\infty)}_{\eps_k}(u_{\eps_k}) 
  &\ge \int_{\Omega \setminus (A_k \cup E^M_k)} 
    f^{(\infty)}_{\eps_k} (x, \nabla v_k(x)) \, dx 
    - \beta |A_k \cup E^M_k| \\ 
  &\ge \int_{\Omega \setminus E^M_k} 
    f^{(\infty)}_{\eps_k} (x, \nabla v_k(x)) \, dx 
    - |A_k \setminus E^M_k| \beta(M^p + 1) - \beta |A_k \cup E^M_k| \\ 
  &\ge \int_{\Omega \setminus E^M_k} f^{(j)}_{\eps_k} (x, \nabla v_k(x)) \, dx 
    - \int_{\Omega} \sup_{|X| \le M} |f^{(j)}_{\eps_k} (x, X) - f^{(\infty)}_{\eps_k} (x, X)| \, dx \\ 
  &\qquad - |A_k | \beta(M^p + 2) - \beta |E^M_k| \\
  &\ge F^{(j)}_{\eps_k}(v_k) 
    - \int_{E^M_k} \beta(|\nabla v_k(x)|^p + 1) \, dx 
   - \int_{\Omega} \sup_{|X| \le M} |f^{(j)}_{\eps_k} (x, X) - f^{(\infty)}_{\eps_k} (x, X)| \, dx \\ 
  &\qquad - |A_k | \beta(M^p + 2) - \beta |E^M_k|. 
\end{align*}
As $ ( |\nabla v_k|^p )_{ k \in \N } $ is equiintegrable, we may for given $\eta > 0$ choose $M$ so large that $\beta \int_{E^M_k} (|\nabla v_k|^p + 2) \, dx \le \eta$ for all $k$. 
Letting first $k \to \infty$ we find that 
\begin{align*}
  F^{(\i)}_{0}(u) 
  &= \lim_{k \to \infty} F^{(\infty)}_{\eps_k}( u_{\eps_k} ) \\ 
  &\ge \limsup_{k \to \infty} F^{(j)}_{\eps_k}(v_k) - \eta 
  - \limsup_{k \to \infty} \int_{\Omega} \sup_{|X| \le M} |f^{(j)}_{\eps_k} (x, X) - f^{(\infty)}_{\eps_k} (x, X)| \, dx. 
\end{align*}
Now let $j \to \infty$. Due to our equivalence assumption and since $\eta > 0$ was arbitrary, we indeed arrive at 
\begin{align*}
  F^{(\i)}_{0}(u) 
  \ge \limsup_{j \to \infty} \limsup_{k \to \infty} F^{(j)}_{\eps_k}(v_k) 
  \ge \limsup_{j \to \infty} F^{(j)}_{0}(u) 
\end{align*}
by the $\liminf$-inequality for $\Gamma(L^p)\text{-}\lim_{k \to \infty} F^{(j)}_{\eps_k} = F^{(j)}_{0}$.  
\item
Step 2: Lower bound. We claim that 
\begin{align}\label{eq:G-lowerbound}
  \liminf_{j \to \infty} F^{(j)}_{0}(u_j) 
  \ge F^{(\i)}_{0}(u) 
\end{align}
whenever $u_j \to u$ in $L^p(\Omega; \R^m)$. 
\item
To prove this, we pass to a subsequence $ (j_k)_{k \in \N} $ such that 
\[ \lim_{k \to \infty} F^{(j_k)}_{0}( u_{j_k} ) = \liminf_{j \to \infty} F^{(j)}_{0}(u_j) \]
and may without loss of generality assume that $ ( F^{(j_k)}_{0}( u_{j_k} ) )_{ k \in \N } $ is bounded and that $j_k$ is chosen so large that 
\[ \limsup_{\eps \to 0} 
   \int_{\Omega} \sup_{|X| \le k} |f^{(j_k)}_{\eps}(x, X) - f^{(\infty)}_{\eps}(x, X)| \, dx 
   < \tfrac{1}{k} .\] 
Then we choose $\eps_{j_k}$ (with $\eps_{j_k} \searrow 0$) so small that there is a $w_{j_k} \in L^p(\Omega; \R^m)$ with 
\[ \| w_{j_k} - u_{j_k} \|_{L^p} \le \tfrac{1}{j_k}
	\quad \mbox{and} \quad
   	F^{(j_k)}_{0}(u_{j_k}) + \tfrac{1}{j_k} 
   	\ge F^{(j_k)}_{\eps_{j_k}} (w_{j_k}) \]
as well as 
\begin{align*}
  \int_{\Omega} \sup_{|X| \le k} |f^{(j_k)}_{\eps_{j_k}} (x, X) - f^{(\infty)}_{\eps_{j_k}} (x, X)| \, dx
  &\le \tfrac{1}{k} + \limsup_{\eps \to 0} \int_{\Omega} 
  \sup_{|X| \le k} |f^{(j_k)}_{\eps} (x, X) - f^{(\infty)}_{\eps} (x, X)| \, dx 
  \le \tfrac{2}{k}.
\end{align*} 
By equicoercivity we may choose a further subsequence $ (j_{k_i} )_{ i \in \N } $ 
such that by Lemma \ref{lemma:equi-int} there are $ v_{i} \in W^{1, p}(\Omega; \R^m) $ 
with $|A_i| \to 0$ for $A_i := \{x \in \Omega : \nabla w_{j_{k_i}}(x) \ne \nabla v_{i}(x) \}$ and $v_{i} \weakly u$ in $W^{1, p}(\Omega; \R^m)$ such that $(| \nabla v_i|^p)_{ i \in \N } $ is equiintegrable. 
Then for $E^M_i := \{x \in \Omega : |\nabla v_i(x)| \ge M \}$ 
\begin{align*}
  F^{(j_{k_i})}_{\eps_{j_{k_i}}} (w_{j_{k_i}}) 
  &= \int_{\Omega} f^{(j_{k_i})}_{\eps_{j_{k_i}}} (x, \nabla w_{j_{k_i}}) \\ 
  &\ge \int_{\Omega \setminus (A_i \cup E^M_i)} f^{(j_{k_i})}_{\eps_{j_{k_i}}} (x, \nabla v_{i}(x)) \ dx 
   - \beta |A_i \cup E^M_i| \\ 
  &\ge \int_{\Omega \setminus E^M_i} f^{(j_{k_i})}_{\eps_{j_{k_i}}} (x, \nabla v_{i}(x)) \ dx 
   - |A_i \setminus E^M_i| \beta(M^p + 1) - \beta |A_i \cup E^M_i| \\ 
  &\ge \int_{\Omega \setminus E^M_i} f^{(\infty)}_{\eps_{j_{k_i}}} (x, \nabla v_{i}(x)) \ dx 
   - \int_{\Omega} \sup_{|X| \le M} |f^{(j_{k_i})}_{\eps_{j_{k_i}}} (x, X) - f^{(\infty)}_{\eps_{j_{k_i}}} (x, X)| \ dx \\ 
  &\qquad - |A_i | \beta(M^p + 2) - \beta |E^M_i| \\ 
  &\ge F^{(\infty)}_{\eps_{j_{k_i}}}(v_{i}) - \int_{E^M_i} \beta( | \nabla v_{i} |^p + 1 ) \ dx  
  - \tfrac{2}{k_i} - |A_i | \beta(M^p + 2) - \beta |E^M_i|, 
\end{align*}
if $k_i \ge M$. For given $\eta > 0$ we find as before by choosing $M$ large and then letting $i \to \infty$ 
\begin{align*}
  \lim_{i \to \infty} F^{(j_{k_i})}_{0}( u_{j_{k_i}} ) 
  \ge \limsup_{i \to \infty} F^{(j_{k_i})}_{\eps_{j_{k_{i}}}}( w_{j_{k_i}} ) 
  &\ge \limsup_{i \to \infty} F^{(\infty)}_{\eps_{j_{k_{i}}}}(v_{i}) - \eta 
\end{align*}
and so, since $v_i$ converges to $u$ in $ L^{p}( \Omega; \R^{m} ) $ and $\eta$ was arbitrary,  
\begin{align*} 
  \liminf_{j \to \infty} F^{(j)}_{0}( u_j ) 
  = \lim_{i \to \infty} F^{(j_{k_i})}_{0}( u_{j_{k_i}} ) 
  \ge F^{(\i)}_{0}(u)
\end{align*}
as claimed. 
\item
Combining \eqref{eq:G-upperbound} and \eqref{eq:G-lowerbound} with $u_j = u$ for all $j$ we arrive at 
\begin{align}\label{eq:GlimFj} 
  \lim_{j \to \infty} F^{(j)}_{0}(u) 
  = F^{(\i)}_{0}(u). 
\end{align}
\item
Step 3: Justification of our assumption.
\item
If we do not assume a priori that $ F^{(\infty)}_{\eps} $ $\Gamma$-converges to $ F^{(\i)}_{0} $, by Theorem \ref{theo:standard-growth-Gamma-compactness} for every subsequence $ ( \eps_k )_{k \in \N} $ there exists a further subsequence $ ( \eps_{k_i} )_{ i \in \N }$ such that 
\[ \Gamma(L^p)\text{-}\lim_{i \to \infty} F^{(\infty)}_{\eps_{k_i}}
   =: F^{(\infty)}_{0} \] 
exists with finite values precisely on $W^{1, p}(\Omega; \R^m)$. Proceeding as above we infer from \eqref{eq:GlimFj} that $F^{(\infty)}_{0}$ does not depend on the particular subsequence $ ( \eps_{k_i} )_{i \in \N} $. Employing the Urysohn property for $\Gamma$-limits we thus find that indeed $\Gamma(L^p)\text{-}\lim_{\eps \to 0} F^{(\infty)}_{\eps}
   = F^{(\infty)}_{0}.$
\end{proof}

\begin{proof}[Proof of Theorem \ref{theo:Gamma-closure-variable}]
Applying Theorem \ref{theo:Gamma-closure-single} to the functionals $F^{(j)}_{\eps}(\cdot, U)$ with $U \in {\cal A}(\Omega)$ fixed we find for each $j \in \N \cup\{\infty\}$ a uniquely determined functional $\Gamma(L^p)\text{-}\lim_{\eps \to 0} F^{(j)}_{\eps}(\cdot, U) =: G^{(j)}_U$. On the other hand, by Theorem \ref{theo:standard-growth-Gamma-compactness} for each $j$ there exists a subsequence $(\eps_k)_{k \in \N}$ such that  
\[ G^{(j)}_U 
   = \Gamma(L^p)\text{-}\lim_{k \to \infty} F^{(j)}_{\eps_k}(\cdot, U) 
   = F^{(j)}_0(\cdot, U) \quad \forall\, U \in {\cal A}(\Omega) \]
for an integral functional $F^{(j)}_0 : L^p(\Omega; \R^m) \times {\cal A}(\Omega) \to \R \cup \{\infty\}$ with  Carath\'{e}odory density $f^{(j)}_0$, whence $G^{(j)}_U = F^{(j)}_0(\cdot, U)$ for all $U$. By  Theorem \ref{theo:Gamma-closure-single} we indeed have 
\[ F^{(\infty)}_0(u, U) 
   = \lim_{j \to \infty} F^{(j)}_0(u, U) 
   = \Gamma(L^p)\text{-}\lim_{j \to \infty} F^{(j)}_0(u, U) \] 
for $ u \in L^{p}( \Omega ; \R^{m} )$ and $U \in {\cal A}(\Omega)$. 

Being continuous in the second argument, the densities $ f^{(j)}_0$ are uniquely determined almost everywhere on $\Omega$ and everywhere on $ \M{m}{n} $ by
\[ \int_U f^{(j)}_0(x, X) \, dx 
   = F^{(j)}_0(\ell_X, U) \quad \mbox{for every } U \in {\cal A}(\Omega), \] 
where $ \ell_X(x) := Xx$ with $X \in \M{m}{n}$. The pointwise convergence also yields for every $X \in \M{m}{n}$ 
\begin{align*}
  \lim_{j \to \infty} \int_U f^{(j)}_0(x, X) \, dx 
  = \lim_{j \to \infty} F^{(j)}_0( \ell_X , U) 
  = F^{(\infty)}_0( \ell_X , U) 
  = \int_{U} f^{(\infty)}_0(x, X) \, dx 
\end{align*}
for all $U \in {\cal A}(\Omega)$. Since $ ( f^{(j)}_0 (\cdot, X) )_{ j \in \N }$ is uniformly bounded 
it follows $f^{(j)}_0 (\cdot, X) \weakstar f^{(\infty)}_0(\cdot, X)$ in $L^{\infty}(\Omega)$. 
\end{proof}

\subsection{$\Gamma$-closure for G{\r a}rding type functionals}

For many interesting applications as, e.g., the ones to be discussed in Section \ref{section:Elasticity}, a two-sided $p$-growth assumption is too restrictive. In this section we generalize our $\Gamma$-closure theorem to integral functionals of `G{\r a}rding type'. More precisely, while imposing $p$-growth assumptions from above as before, the integral densities will only assumed to be lower bounded by some constant. Yet the functionals are still supposed to satisfy a weak coercivity assumption, which we impose by requiring a G{\r a}rding type inequality to hold. 

\begin{definition}\label{definition:Garding} 
We say that the family of integral functionals $F^{(j)}_{\eps}$, $j \in \N \cup \{\infty\}$, $\eps > 0$, with densities $f^{(j)}_{\eps} : \Omega \times \M{m}{n} \to \R$ is of uniform $p$-G{\r a}rding type on $U \subset \Omega$ open, $1 < p < \infty$, if there are  $ \alpha, \beta > 0 $ independent of $j$ and $\eps$ such that the $f^{(j)}_{\eps}$ satisfy 
\begin{align*}
  - \beta \le f^{(j)}_{\eps}(x,X) \le \beta( |X|^{p} + 1 ) 
\end{align*}
for almost all $x \in \Omega $ and all $X \in \M{m}{n}$ and moreover there are $\alpha_U, \gamma_U$ such that 
\[ F^{(j)}_{\eps}(u) 
   \ge \alpha_U \int_{U} |\nabla u(x)|^p \, dx - \gamma_U \int_{U} |u(x)|^p \, dx \]
for all $u \in W^{1, p}(U; \R^m)$. 
\end{definition}

\begin{theo}[$\Gamma$-closure on a single domain]\label{theo:Gamma-closure-single-Garding}
Let $\Omega \subset \R^n$ be bounded and open. Suppose that the family of functionals $F^{(j)}_{\eps}$, $ j \in \N \cup \{ \infty \}$, $ \eps > 0$, 
with densities $f^{(j)}_{\eps} : \Omega \times \M{m}{n} \to \R$ is of uniform $p$-G{\r a}rding type on $\Omega$. Assume that
\begin{itemize} 
\item[(i)] For each $j < \infty$ the $\Gamma$-limit $\Gamma(L^p)\text{-}\lim_{\eps \to 0} F^{(j)}_{\eps} =: F^{(j)}_0$ exists. 

\item[(ii)] The families $((f^{(j)}_{\eps})_{\eps > 0})_{j \in \N}$ and $(f^{(\infty)}_{\eps})_{\eps > 0}$ are equivalent on $\Omega$. 
\end{itemize}

Then also $\Gamma(L^p)\text{-}\lim_{\eps \to 0} F^{(\infty)}_{\eps} =: F^{(\infty)}_0$ exists and is the pointwise and the $\Gamma$-limit of $F^{(j)}_0$ as $j \to \infty$: 
\[ F^{(\infty)}_0 
   = \lim_{j \to \infty} F^{(j)}_0 
   = \Gamma(L^p)\text{-}\lim_{j \to \infty} F^{(j)}_0. \] 
\end{theo}

\begin{remark}\label{remark:no-upper}
In fact, the assumption that the $f^{(j)}_{\eps}$ be bounded from above can be dropped. In order to see this, it suffices to combine Lemma \ref{lemma:equi-int} with \cite[Proposition A.1]{FrieseckeJamesMueller} so as to obtain approximations with uniformly bounded gradients in the proof of Theorem \ref{theo:Gamma-closure-single}. In general one then only has that $F^{(\infty)}_0 = \lim_{j \to \infty} F^{(j)}_0$ on $W^{1, \infty}(\Omega; \R^m)$. By density, however, this is enough to obtain that still $F^{(\infty)}_0 = \Gamma(L^p)\text{-}\lim_{j \to \infty} F^{(j)}_0$. Although of interest in models of elasticity theory, we do not pursue this line of thought here as in our main application to homogenization theory in Section \ref{section:Elasticity} the assumptions are only known to be satisfied under a standard $p$-growth assumption from above, cf.\ Remark \ref{remark:muwe}. 
\end{remark}

Again we also state a version of this result on variable domains as a corollary. Note that here the constants in the G{\r a}rding inequality are allowed to explicitly depend upon the domain $U \subset \Omega$. 
\begin{theo}[$\Gamma$-closure on variable domains]\label{theo:Gamma-closure-variable-Garding}
Let $\Omega \subset \R^n$ be open. Suppose that the family of functionals $F^{(j)}_{\eps}$, $ j \in \N \cup \{ \infty \}$, $ \eps > 0$, 
with densities $f^{(j)}_{\eps} : \Omega \times \M{m}{n} \to \R$ is of uniform $p$-G{\r a}rding type on every $U \in {\cal A}(\Omega)$. Assume that
\begin{itemize} 
\item[(i)] For each $j < \infty$ and $U \in {\cal A}(\Omega)$ the $\Gamma$-limit $\Gamma(L^p)\text{-}\lim_{\eps \to 0} F^{(j)}_{\eps}(\cdot, U)$ exists. 

\item[(ii)] The families $((f^{(j)}_{\eps})_{\eps > 0})_{j \in \N}$ and $(f^{(\infty)}_{\eps})_{\eps > 0}$ are equivalent on every $U \in {\cal A}(\Omega)$. 
\end{itemize}

Then for each $ j \in \N \cup \{ \i \} $ there exists a Carath{\'e}odory function $ f^{(j)}_{0} : \Omega \times \M{m}{n} \to \R $, uniquely determined a.e.\ on $\Omega$, 
such that for the corresponding integral functional $F^{(j)}_{0}$ it holds 
\begin{align*}
  \Gamma(L^p)\text{-}\lim_{\eps \to 0} F^{(j)}_{\eps}(\cdot, U)
   = F^{(j)}_0(\cdot, U)
\end{align*}
for all $U \in {\cal A}(\Omega)$. Moreover, for $ u \in L^{p}( \Omega ; \R^{m} ) $
\[ F^{(\infty)}_0(u, U) 
   = \lim_{j \to \infty} F^{(j)}_0(u, U) 
   = \Gamma(L^p)\text{-}\lim_{j \to \infty} F^{(j)}_0(u, U) \] 
and the limiting densities $ f^{(j)}_{0} $ for all $X \in \M{m}{n}$ satisfy
\[ f^{(j)}_0(\cdot, X) \weakstar f^{(\infty)}_0 (\cdot, X) \quad \mbox{in } L^{\infty}(\Omega). \] 
\end{theo}

The following proposition not only is the first step towards the proofs of Theorems \ref{theo:Gamma-closure-single-Garding} and \ref{theo:Gamma-closure-variable-Garding}. It also provides a criterion for the $\Gamma$-convergence of G{\r a}rding type functionals and might thus be used to verify the assumption (i) in Theorems \ref{theo:Gamma-closure-single-Garding} and \ref{theo:Gamma-closure-variable-Garding} in particular situations. In particular, it shows that the assumptions of these theorems are always satisfied for suitable subsequences. 

\begin{prop}
\label{prop:addition-trick}
Let $\Omega \subset \R^n$ be open. Suppose that the family of functionals $F_{\eps}$, $ \eps > 0$, with densities $f_{\eps} : \Omega \times \M{m}{n} \to \R$ is of uniform $p$-G{\r a}rding type on every $U \in {\cal A}(\Omega)$. 
Let us take some null sequence $ \lambda_{k} \searrow 0 $, define 
\[ f_{\eps}^{ (k) }( x , X ) := f_{\eps}( x , X ) + \lambda_{k} |X|^{p} \] 
and by $F^{(k)}_{\eps}$ denote the corresponding integral functional. 
Assume that for every $k \in \N$ and all $U \in {\cal A}(\Omega)$ the $\Gamma$-limit $\Gamma(L^p)\text{-}\lim_{\eps \to 0} F^{(k)}_{\eps}(\cdot, U) =: F^{(k)}_0(\cdot, U)$ exists. 

Then also $\Gamma(L^p)\text{-}\lim_{\eps \to 0} F_{\eps}(\cdot, U) =: F_0(\cdot, U)$ exists and is given by 
\[ F_{0}(\cdot, U)
   = \inf_{ k \in \N } F_{0}^{(k)}(\cdot, U) 
   = \lim_{ k \to \infty } F_{0}^{(k)}(\cdot, U)
   = \Gamma(L^p)\text{-}\lim_{ k \to \infty } F_{0}^{(k)}(\cdot, U). \] 
Moreover, $F_0$ and $F^{(k)}_0$, $j \in \N$, are given in terms of Carath\'{e}odory integral densities $f_0$ and $f^{(k)}_0$, respectively, 
such that for a.e.\ $x \in \Omega$ and all $X \in \M{m}{n}$ 
\[ f_0(x, X) = \inf_{k \in \N} f_0^{(k)}(x, X) = \lim_{k \in \N} f_0^{(k)}(x, X). \] 
\end{prop} 

\begin{proof}
Since $\| \nabla u \|_{L^{p}(U)}^{p} \le \alpha_U^{-1} (F_{\eps}(u) + \gamma_U \|u\|_{L^{p}(U)}^{p})$ for $u \in W^{1, p}(U; \R^m)$, we obtain 
\[ \Big( 1 - \frac{ \lambda_{k} }{ \alpha_{U} } \Big) 
   F_{\eps}^{ (k) }(u) - \frac{ \gamma_{U} \lambda_{k} }{ \alpha_{U} } \| u \|_{ L^{p}(U)}^{p} 
   \le F_{\eps}(u) 
   \le F_{\eps}^{ (k) }(u) \]
for all $u \in L^{p}(U; \R^{m})$. As the $\Gamma$-$\liminf$ and the $\Gamma$-$\limsup$ are stable under continuous perturbations, this implies that 
\begin{align}
\label{eq:addition-trick-1}
  \Big( 1 - \frac{\lambda_{k}}{\alpha_{U}} \Big) 
  F_{0}^{(k)}(u) - \frac{\gamma_{U} \lambda_{k}}{\alpha_{U} } \|u\|_{L^{p}(U)}^{p} 
  \le \Gamma(L^p)\text{-}\liminf_{\eps \to 0} F_{\eps}(u) 
  \le \Gamma(L^p)\text{-}\limsup_{\eps \to 0} F_{\eps}(u) 
  \le F_{0}^{ (k) }(u) 
\end{align}
and so, due to monotonicity in $k$, 
\begin{equation}
\label{eq:addition-trick-2}
  \Gamma(L^p)\text{-}\lim_{\eps \to 0} F_{\eps}(u) 
  = \inf_{k \in \N} F_{0}^{ (k) }(u) 
  = \lim_{k \to \infty} F_{0}^{ (k) }(u) 
  = \Gamma(L^p)\text{-}\lim_{k \to \i} F_{0}^{ (k) }, 
\end{equation}
where last equality follows from the fact that $ F_{0}^{ (k) } $ is decreasing in $k$ and that $ F_{0} $ is lower semicontinuous.  

In order to prove the statement on the densities, we first note that by Theorem \ref{theo:standard-growth-Gamma-compactness} there exist Carath\'{e}odory functions $f^{(k)}_0$ such that $F^{(k)}_0(u, U) = \int_U f^{(k)}_0(x, \nabla u(x)) \, dx$ for $u \in W^{1, p}(U; \R^m)$ and $+ \infty$ otherwise. By monotonicity in $k$ 
\[ \int_U f^{(k')}_0(x, X) \, dx 
   \le \int_U f^{(k)}_0(x, X) \, dx \] 
for all $X \in \M{m}{n}$ and $U \in {\cal A}(\Omega)$ if $k' \ge k$. Because of continuity in $X$ it follows for a.e.\ $x \in \Omega$ 
\[ f^{(k')}_0(x, X) 
   \le f^{(k)}_0(x, X) \quad \forall \, X \in \M{m}{n}. \] 
By monotone convergence it thus follows that 
\[ F_0(u, U) 
   = \lim_{k \to \infty} F^{(k)}_0(u, U) 
   = \int_U \lim_{k \to \infty} f^{(k)}_0 (x, \nabla u(x)) \, dx \]
for $u \in W^{1, p}(U; \R^m)$ and $F_0(u, U) = + \infty$ otherwise. As $f^{(k)}_0(x, \cdot)$ is quasiconvex for almost every $x$, so is $f_0(x, \cdot)$, which shows that $f_0$ is Carath{\'e}odory. 
\end{proof}

\begin{remark}
The first part of this proof shows that requiring the uniform $p$-G{\r a}rding assumption and the existence of the $\Gamma$-limits only on a single bounded and open region $\Omega$, one still has that 
\[ \Gamma(L^p)\text{-}\lim_{\eps \to 0} F_{\eps} 
   = \inf_{ j \in \N } F_{0}^{(j)} 
   = \lim_{ j \to \infty } F_{0}^{(j)}. \] 
\end{remark}

We now prove Theorem \ref{theo:Gamma-closure-variable-Garding} by reducing with the help of Proposition \ref{prop:addition-trick} to standard growth conditions. The proof of Theorem \ref{theo:Gamma-closure-single-Garding} will then be a straightforward adaption of the first part of the following proof. 

\begin{proof}[Proof of Theorem \ref{theo:Gamma-closure-variable-Garding}] 
Let for every $ k \in \N $, $ j \in \N \cup \{\i\} $ and $ \eps > 0 $
\[ f^{(k,j)}_{\eps}(x, X) := f^{(j)}_{\eps}(x, X) + \frac{1}{k} |X|^p 
\quad \mbox{and}
\quad
F^{(k,j)}_{\eps}(u, U) := F^{(j)}_{\eps}(u, U) + \frac{1}{k} \int_U |\nabla u(x)|^p \, dx. \]
Clearly, for each (fixed) $ k \in \N $ 
the $f^{(k, j)}_{\eps}$ uniformly satisfy standard $p$-growth assumptions and $((f^{(k, j)}_{\eps})_{\eps > 0})_{j \in \N}$ is equivalent to $(f^{(k, \infty)}_{\eps})_{\eps > 0}$. 

Assume that for every $k \in \N $, $j \in \N $, and all $U \in {\cal A}(\Omega)$ the $\Gamma$-limit of $F^{(k, j)}_{\eps}$ exists and is given by 
\[ \Gamma(L^p)\text{-}\lim_{\eps \to 0} F^{(k, j)}_{\eps}(\cdot, U) 
   = F^{(k,j)}_0(\cdot, U), \]
where $F^{(k,j)}_0(\cdot, U)$ is an integral functional with density $f^{(k,j)}_0$. Then by Theorem \ref{theo:Gamma-closure-variable}
there exist Borel functions $f^{(k, \infty)}_{0} : \Omega \times \M{m}{n} \to \R$ and corresponding integral functionals $F^{(k, \infty)}_{0}$ such that 
\begin{align*}
  \Gamma(L^p)\text{-}\lim_{\eps \to 0} F^{(k, \infty)}_{\eps}(\cdot, U)
   = F^{(k, \infty)}_0(\cdot, U)
\end{align*}
for all $U \in {\cal A}(\Omega)$ and 
\begin{align*}
  \Gamma(L^p)\text{-}\lim_{j \to \infty} F^{(k, j)}_{0}(u, U)
  = \lim_{j \to \infty} F^{(k, j)}_{0}(u, U)
  = F^{(k, \infty)}_0(u, U)
\end{align*}
for all $u \in L^p(\Omega; \R^m)$ and $U \in {\cal A}(\Omega)$. 

From Proposition~\ref{prop:addition-trick} it follows immediately that for every $ j \in \N \cup \{ \infty \} $
\begin{align*}
  \Gamma(L^p)\text{-}\lim_{\eps \to 0} F^{(j)}_{\eps}(\cdot, U)
   = F^{(j)}_0(\cdot, U)
   := \inf_{ k \in \N } F^{(k, j)}_0(\cdot, U) 
\end{align*}
and it remains to prove that 
\[ F^{(\infty)}_0(\cdot, U) = \lim_{j \to \i} F^{(j)}_0(\cdot, U) = \Gamma(L^p)\text{-}\lim_{j \to \i} F^{(j)}_0(\cdot, U). \]
To this end, we first infer from \eqref{eq:addition-trick-2} that 
\[
  \limsup_{ j \to \i } F^{(j)}_{0}(u, U)
  \le \inf_{ k \in \N } \limsup_{ j \to \i } F^{(k, j)}_{0}(u, U) 
  = \inf_{ k \in \N } F^{(k, \i)}_{0}(u, U)  
  = F^{(\i)}_{0}(u, U).
\]
By \eqref{eq:addition-trick-1} on the other hand, for any $ k \in \N $,
\begin{align*}
  \liminf_{ j \to \i } F^{(j)}_{0}(u, U)
  &\ge \liminf_{ j \to \i } \left( 1 - \frac{ 1 }{ \alpha_{U} k } \right) F_{0}^{( k , j )}(u,U) 
	- \frac{ \gamma_{U} }{ \alpha_{U} k } \| u \|_{ L^{p}( U ) }^{p} \\
  &= \left( 1 - \frac{ 1 }{ \alpha_{U} k } \right) F_{0}^{( k , \i )}(u,U) 
	- \frac{ \gamma_{U} }{ \alpha_{U} k } \| u \|_{ L^{p}( U ) }^{p} \\
  &\ge \left( 1 - \frac{ 1 }{ \alpha_{U} k } \right) F_{0}^{(\i)}(u,U) 
	- \frac{ \gamma_{U} }{ \alpha_{U} k } \| u \|_{ L^{p}( U ) }^{p}
\end{align*}
and so
\[ \liminf_{ j \to \i } F^{(j)}_{0}(u, U) \ge F_{0}^{(\i)}(u,U). \]
Finally we note that exactly the same may be done for the $ \Gamma\text{-}\liminf $ and $ \Gamma\text{-}\limsup $ in place of $\liminf$, respectively, $\limsup$. 

Now if we do not a priori assume that the $\Gamma$-limits of $F^{(k, j)}_{\eps}$ exist, a diagonal sequence argument shows that for any subsequence $\eps_i$ there is a further subsequence $\eps_{i_l}$ such that 
\[ \Gamma(L^p)\text{-}\lim_{l \to \infty} F^{(k, j)}_{\eps_{i_l}}(\cdot, U) 
   = F^{(k,j)}_0(\cdot, U) \]
for all $ k \in \N $, $j \in \N \cup\{\infty\}$ and $U \in {\cal A}(\Omega)$, where $F^{(k,j)}_0(\cdot, U)$ is an integral functional with density $f^{(k,j)}_0$. Then, as shown above, $F^{(\infty)}_0(u, U)$ is independent of the subsequence chosen, so that in fact 
\begin{align*}
  \Gamma(L^p)\text{-}\lim_{\eps \to 0} F^{(\infty)}_\eps(\cdot, U) 
  = F^{(\infty)}_0(\cdot, U) 
\end{align*}
for $F^{(\infty)}_0(u, U) = \lim_{j \to 0} F^{(j)}_{0}(u, U)$. 

Finally, the convergence of the densities now follows precisely as in the proof of Theorem \ref{theo:Gamma-closure-variable}. 
\end{proof}

\begin{proof}[Proof of Theorem \ref{theo:Gamma-closure-single-Garding}] This is follows exactly along the lines of the first part of the proof of Theorem \ref{theo:Gamma-closure-variable-Garding} taking into account the remark after Proposition \ref{prop:addition-trick}. 
\end{proof}

\subsection{Boundary values and compactness}

In this section we will first prove that the $\Gamma$-closure theorem for G{\r a}rding type functionals remains true for functionals with prescribed boundary values. On the other hand, G{\r a}rding type functionals may lack coercivity so that bounded energy sequences do not necessarily admit convergent subsequences. We will see, however, that this lack of compactness may be circumvented on suitable domains by imposing boundary values. 

Let us for this subsection fix a bounded and open set $\Omega \subset \R^n$ and a function $u_0 \in W^{1, p}(\Omega; \R^m)$. For an integral functional $F$ we denote by 
\[ \bar{F}(u) 
   := \begin{cases} 
        F(u), &\text{if } u \in u_0 + W^{1, p}_0(\Omega; \R^m), \\ 
        + \infty &\text{otherwise}
     \end{cases} \]
its restriction to prescribed boundary values $u_0$ on $\partial \Omega$. 

\begin{theo}[$\Gamma$-closure with boundary values]\label{theo:Gamma-closure-boundary-values}
Suppose that the family of functionals $F^{(j)}_{\eps}$, $ j \in \N \cup \{ \infty \}$, $ \eps > 0$, with densities $f^{(j)}_{\eps} : \Omega \times \M{m}{n} \to \R$ is of uniform $p$-G{\r a}rding type on $\Omega$. Assume that
\begin{itemize} 
\item[(i)] For each $j < \infty$, $\Gamma(L^p)\text{-}\lim_{\eps \to 0} F^{(j)}_{\eps} =: F^{(j)}_0$ exists. 

\item[(ii)] The families $((f^{(j)}_{\eps})_{\eps > 0})_{j \in \N}$ and $(f^{(\infty)}_{\eps})_{\eps > 0}$ are equivalent on $\Omega$. 
\end{itemize}

Then also $\Gamma(L^p)\text{-}\lim_{\eps \to 0} F^{(\infty)}_{\eps} =: F^{(\infty)}_0$ exists, 
\[ \Gamma(L^p)\text{-}\lim_{\eps \to 0} \bar{F}^{(j)}_{\eps} =: \bar{F}^{(j)}_0 
   \quad \forall j \in \N \cup \{+\infty\} 
   \quad\text{and}\quad 
   \bar{F}^{(\infty)}_0 
   = \lim_{j \to \infty} \bar{F}^{(j)}_0 
   = \Gamma(L^p)\text{-}\lim_{j \to \infty} \bar{F}^{(j)}_0. \] 
\end{theo}

We first show that prescribing boundary conditions is compatible with taking $\Gamma$-limits for a single functional satisfying a G{\r a}rding type inequality.
\begin{lemma}[Boundary values for G{\r a}rding type functionals]\label{lemma:boundary-values}
Let $\Omega$ have a Lipschitz boundary. Suppose $f_{\eps} : \Omega \times \M{m}{n} \to \R$ are Borel functions such that for some $\beta > 0$ and $1 < p < \infty$
\[ - \beta \le f_{\eps}(x, X) \le \beta( |X|^{p} + 1 ) \]
for all $x \in \Omega$, $X \in \M{m}{n}$ and $\eps > 0$ and there are $\alpha, \gamma > 0$ such that 
\[ F_{\eps}(u) 
   \ge \alpha \int_{ \Omega } |\nabla u(x)|^p \, dx - \gamma \int_{ \Omega } |u(x)|^p \, dx \]
for all $ u \in W^{1, p}( \Omega ; \R^m ) $, $\eps > 0$. Then $\Gamma(L^p)\text{-}\lim_{\eps \to 0} F_{\eps} = F_{0}$ implies $\Gamma(L^p)\text{-}\lim_{\eps \to 0} \bar{F}_{\eps} = \bar{F}_{0}$. 
\end{lemma}

\begin{proof}
1. $\liminf$-inequality. Suppose $u_{\eps} \to u$ in $L^{p}(\Omega ; \R^{m})$. If $\bar{F}_{0}(u) = \infty$, then indeed $\liminf_{\eps \to 0} \bar{F}_{\eps}(u_{\eps}) = \infty$ for otherwise we may by G\r{a}rding's inequality find a subsequence $(u_{\eps_{k}})_{k \in \N} \subset u_{0} + W_{0}^{1,p}(\Omega; \R^{m})$ with $ u_{\eps_{k}} \weakly u$ in $ W^{1,p}(\Omega ; \R^{m})$. This contradicts $\bar{F}_{0}(u) = \infty$. If, on the other hand, $\bar{F}_{0}(u) < \infty$, then 
\[ \liminf_{\eps \to 0} \bar{F}_{\eps}(u_{\eps}) 
   \ge \liminf_{\eps \to 0} F_{\eps}(u_{\eps}, \Omega)
   \ge F_{0}(u, \Omega) = \bar{F}_{0}(u). \]

2. Consider an arbitrary subsequence $\eps_{j} \searrow 0$. Referring to the pointwise Urysohn property from Theorem \ref{theo:Urysohn}, for given $u \in L^p(\Omega; \R^m)$ we only have to provide a recovery sequence $v_k \to u$ in $L^p(\Omega; \R^m)$ along a suitable subsequence $(\eps_{j_k})_{k \in \N}$ of $(\eps_{j})_{j \in \N}$. 

In order to do so, we start with a recovery sequence $u_{\eps}$ for the original functional $F_{0}$ such that $u_{\eps} \to u$ in $L^p(\Omega; \R^m)$ and $\lim_{\eps \to 0} F_{\eps}(u_{\eps}) = F_0(u)$. If $F_0(u) = \infty$, then also $\lim_{\eps \to 0} \bar{F}_{\eps}(u_{\eps}) = \bar{F}_0(u)$ and the claim follows. If $F_0(u) < \infty$, then by G{\r a}rding's inequality $(u_{\eps})_{\eps > 0}$ is bounded in $W^{1, p}(\Omega; \R^m)$. Referring to Lemma \ref{lemma:equi-int} we find a subsequence $(\eps_{j_k})_{k \in \N}$ and $v_k \in u_{0} + W^{1,p}_{0}(\Omega; \R^{m})$ such that $(|\nabla v_{k}|^{p})_{k \in \N}$ is equiintegrable in $\Omega$, $|A_k| \to 0$ for 
\[ A_{k} := \{ u_{ \eps_{j_{k}} } \ne v_{k} \mbox{ or } \nabla u_{ \eps_{j_{k}} } \ne \nabla v_{k} \} \] 
and $v_k \weakly u$ in $W^{1, p}(\Omega; \R^m)$. But then
\begin{align*} 
  \bar{F}_{\eps_{j_{k}}}(v_{k})
  &=  F_{\eps_{j_{k}}}(v_{k}) \\
  &= \int_{\Omega \setminus A_{k}} f_{\eps_{j_{k}}}(x, \nabla u_{\eps_{j_{k}}}(x)) \, dx 
     + \int_{A_{k}} f_{\eps_{j_{k}}}(x, \nabla v_{k}(x)) \, dx \\
  &\le F_{ \eps_{ j_{k} } }( u_{ \eps_{ l_{k} } }) + \int_{ A_{k} } \beta( 1 + | \nabla v_{k}(x) |^{p} ) \ dx 
\end{align*}
and thus 
\[ \limsup_{k \to \infty} \bar{F}_{\eps_{j_{k}}}(v_{k}) 
   \le \limsup_{k \to \infty} F_{\eps_{j_{k}}}(u_{\eps_{j_{k}}}) 
   \le F_{0}(u) = \bar{F}_{0}(u). \qedhere \]
\end{proof}

\begin{proof}[Proof of Theorem \ref{theo:Gamma-closure-boundary-values}.] 
$\Gamma(L^p)\text{-}\lim_{\eps \to 0} F^{(\infty)}_{\eps} =: F^{(\infty)}_0$ was shown in Theorem \ref{theo:Gamma-closure-single-Garding}. The remaining assertions follow from the same Theorem and Lemma \ref{lemma:boundary-values} by noting that also the family $F^{(j)}_0$, $j \in \N \cup \{\infty\}$, is of uniform $p$-G{\r a}rding type on $\Omega$ with the same constants $\alpha_{\Omega}$ and $\gamma_{\Omega}$: For given $j$ let $u_{\eps}$ be a recovery sequence for $u \in W^{1,p}(\Omega; \R^m)$. By G\r{a}rding's inequality and boundedness of $\big(F^{(j)}_{\eps}(u_{\eps})\big)_{\eps > 0}$ we have $u_{\eps} \weakly u$ in $ W^{1,p}(\Omega;\R^{m})$ and so  
\begin{align*} 
  F^{(j)}_0(u) 
  &= \lim_{\eps \to 0} F^{(j)}_{\eps}(u_{\eps}) \\
  &\ge \liminf_{\eps \to 0} \Big(\alpha_{\Omega} \| \nabla u_{\eps} \|_{ L^{p}(\Omega)}^{p} 
   - \gamma_{\Omega} \| u_{ \eps } \|_{ L^{p}(\Omega)}^{p} \Big) \\
  &\ge \alpha_{\Omega} \| \nabla u \|_{ L^{p}(\Omega) }^{p} 
   - \gamma_{\Omega} \| u \|_{ L^{p}(\Omega)}^{p}. \qedhere
\end{align*}
\end{proof}

In view of our application to homogenization theory to be discussed below, however, we observe that Poincar{\'e}'s inequality guarantees coercivity on sufficiently small domains. 

\begin{prop}\label{prop:coercivity-small-domains}
Suppose $f : \Omega \times \M{m}{n} \to \R$ is a Borel function such that for some $\beta > 0$ and $1 < p < \infty$
\[ - \beta \le f(x, X) \le \beta( |X|^{p} + 1 ) \]
for all $x \in \Omega$, $X \in \M{m}{n}$ and there are $\alpha_{\Omega}, \gamma_{\Omega} > 0$ such that 
\[ F(u, \Omega) 
   \ge \alpha_{\Omega} \int_{\Omega} |\nabla u(x)|^p \, dx - \gamma_{\Omega} \int_{\Omega} |u(x)|^p \, dx \]
for all $u \in W^{1, p}(\Omega; \R^m)$. Then, if $U \subset \Omega$ is sufficiently small, there are constants $A > 0$ and $B$ such that
\[ F(u, U) 
   \ge A \int_{U} |\nabla u(x)|^p \, dx - B \] 
for all $u \in u_0 + W^{1, p}_0(U; \R^m)$. 
\end{prop}

\begin{proof}
Take any $ U \in {\cal A}( \Omega ) $. For $u \in u_0 + W^{1, p}_0(U; \R^m)$ let $\bar{u} \in W^{1, p}(\Omega; \R^m)$ be its extension by $u_0$ on $\Omega \setminus U$. We have 
\begin{align*}
  F(u, U) 
  &= F(\bar{u}, \Omega) - \int_{\Omega \setminus U} f(x, \nabla u_0(x)) \, dx \\ 
  &\ge \int_{\Omega} \Big( \alpha_{\Omega} |\nabla \bar{u}(x))|^p - \gamma_{\Omega} |\bar{u}(x)|^p \Big) \, dx 
       - \int_{\Omega \setminus U} \Big( \beta |\nabla u_0(x)|^p + \beta \Big) \, dx \\ 
  &= \int_{U} \Big( \alpha_{\Omega} |\nabla u(x))|^p - \gamma_{\Omega} |u(x)|^p \Big) \, dx 
       - \int_{\Omega \setminus U} \Big( (\beta - \alpha_{\Omega}) |\nabla u_0(x)|^p 
       + \gamma_{\Omega} |u_0(x)|^p + \beta \Big) \, dx. 
\end{align*}
Now if $C_{\rm P}$ denotes the Poincar{\'e} constant of $U$, then 
\begin{align*}
  \int_U |u(x)|^p \, dx 
  &\le \int_U \Big( 2^p |u(x) - u_0(x)|^p + 2^p |u_0(x)|^p \Big) \, dx \\ 
  &\le 2^p C_{\rm P}^p \int_U |\nabla u(x) - \nabla u_0(x)|^p \, dx 
     + 2^p \int_U |u_0(x)|^p \, dx \\ 
  &\le 4^p C_{\rm P}^p \int_U |\nabla u(x)|^p \, dx 
     + 4^p C_{\rm P}^p \int_U |\nabla u_0(x)|^p \, dx 
     + 2^p \int_U |u_0(x)|^p \, dx 
\end{align*} 
and the assertion follows with $:= \alpha_{\Omega} - 4 C_{\rm P}^p$ which is positive if $U$ is sufficiently small. 
\end{proof}

The proof shows that $A$ and $B$ only depend on the width of $U$ (through its Poincar\'{e} constant), $ p, \alpha_{\Omega} , \gamma_{\Omega} , \beta $ and $ \| u_{0} \|_{ W^{1,p}(\Omega) } $.

\section{General applications}\label{section:General-applications}

\subsection{A perturbation and a relaxation result}

When specialized to $j$- or $\eps$-independent families our $\Gamma$-closure theorems immediately imply the following perturbation and relaxation results. We only consider their formulation on a single domain $\Omega$, the adaption to variable domains is straightforward. 

The first easy consequence of the $\Gamma$-closure theorems is a stability result for $\Gamma$-limits under equivalent perturbations of the densities.
\begin{theo}[Perturbation]\label{theo:perturbation}
Let $\Omega \subset \R^n$ be bounded and open. Suppose that the families of functionals $F_{\eps}$ and $G_{\eps}$, $ \eps > 0$, with densities $f_{\eps}, g_{\eps} : \Omega \times \M{m}{n} \to \R$, respectively, are of uniform $p$-G{\r a}rding type on $\Omega$. Assume that $\Gamma(L^p)\text{-}\lim_{\eps \to 0} F_{\eps} = F_0$ and 
\[ \limsup_{\eps \to 0} \int_{\Omega} \sup_{|X| \le R} | f_{\eps}(x, X) - g_{\eps}(x , X) | \, dx = 0 \]
for all $ R > 0 $. 
Then also $\Gamma(L^p)\text{-}\lim_{\eps \to 0} G_{\eps}= F_0$. 
\end{theo}

\begin{figure}[h!]
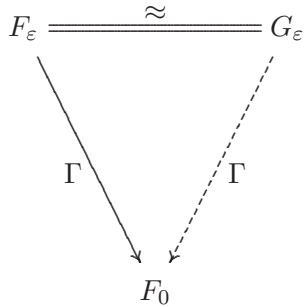

\null\hspace{5.3cm}
\begindc{\commdiag}[500]
\obj(0,2)[aa]{$ F_{ \eps } $}
\obj(2,2)[bb]{$ G_{ \eps } $}
\obj(1,0)[cc]{$ F_{ 0 } $}
\mor{aa}{bb}{$ \approx $}[\atleft,\equalline]
\mor{bb}{cc}{$ \Gamma $}[\atleft,1]
\mor{aa}{cc}{$ \Gamma $}[\atright,0]
\enddc\caption{Schematic diagram of the assumptions (solid lines) and consequences (dashed) in Theorem \ref{theo:perturbation}.}
\end{figure}

\begin{proof}
With $F^{(j)}_{\eps} = F_{\eps}$ for all $j \in \N$ and $\eps \ge 0$ and $G_{\eps} = F^{(\infty)}_{\eps}$ for $\eps > 0$ it follows directly from Theorem \ref{theo:Gamma-closure-single-Garding} that 
\[ \Gamma(L^p)\text{-}\lim_{\eps \to 0} G_{\eps} 
   = \lim_{j \to \infty} F_0 
   = F_0. \qedhere \]
\end{proof}

Our second straightforward application shows that a sequence of functionals equivalent to some fixed functional $\Gamma$-converges to the relaxation of this functional. 
\begin{theo}[Relaxation]\label{theo:relaxation}
Let $\Omega \subset \R^n$ be bounded and open. Suppose that the family of functionals $F^{(j)}$, $j \in \N \cup \{\infty\}$ with densities $f^{(j)} : \Omega \times \M{m}{n} \to \R$ are of uniform $p$-G{\r a}rding type on $\Omega$. Assume that 
\[ \lim_{j \to \infty} \int_{\Omega} \sup_{|X| \le R} | f^{(j)}(x, X) - f^{(\i)}(x , X) | \, dx = 0 \]
for all $ R > 0 $. 
Then $\Gamma(L^p)\text{-}\lim_{j \to \infty} F^{(j)} = \lsc F^{ (\infty) }$. 
\end{theo}

\begin{figure}[h!]
\null\hspace{5.3cm}
\begindc{\commdiag}[500]
\obj(0,2)[aa]{$ F^{ (j) } $}
\obj(2,2)[bb]{$ F^{ (\infty) } $}
\obj(0,0)[cc]{$ \lsc F^{ (j) } $}
\obj(2,0)[dd]{$ \lsc F^{ (\infty) } $}
\mor{aa}{bb}{$ \approx $}[\atleft,\equalline]
\mor{aa}{cc}{$ \Gamma $}[\atleft,1]
\mor{bb}{dd}{$ \Gamma $}[\atleft,1]
\mor{cc}{dd}{$ \Gamma $}[\atleft,1]
\mor{aa}{dd}{$ \Gamma $}[\atleft,1]
\mor{cc}{dd}{ptw.}[\atright,1]
\enddc\caption{\label{fig:relaxation}Schematic diagram of the assumptions (solid lines) and consequences (dashed) in Theorem \ref{theo:relaxation}. }
\end{figure}

\begin{proof}
Let $F^{(j)}_{\eps} = F^{(j)}$ for all $j \in \N \cup \{\infty\}$ and $\eps > 0$. These sequences are constant in $\eps$ and so $\Gamma(L^p)\text{-}\lim_{\eps \to 0} F^{(j)}_{\eps} = \lsc F^{(j)}$. A direct application of Theorem \ref{theo:Gamma-closure-single-Garding} thus yields 
\[ \Gamma(L^p)\text{-}\lim_{j \to \infty} \lsc F^{(j)} 
   = \lim_{j \to \infty} \lsc F^{(j)} 
   = \lsc F^{ (\infty) }. \] 
Noting that $\Gamma(L^p)\text{-}\lim_{j \to \infty} F^{(j)} = \Gamma(L^p)\text{-}\lim_{j \to \infty} \lsc F^{(j)}$ finishes the proof. 
\end{proof}

\subsection{Commutability of $\Gamma$-limits}

In the general case with a doubly indexed family of functionals considered in the closure theorems above and schematically illustrated in Figure \ref{fig:Gamma-closure} the natural questions arise if the $\Gamma$-limits as $\eps \to 0$ and as $j \to \infty$ commute and if the limiting functional is also a simultaneous limit of $F^{(j_k)}_{\eps_k}$ as $k \to \infty$. In general this is not the case as the example 
\[ f^{(j)}_{\eps}(x, X) 
   = \begin{cases} 
       |X|^p, & \text{if } \eps < \tfrac{1}{j}, \\ 
       2|X|^p, & \text{if } \eps \ge \frac{1}{j},
     \end{cases} 
   \quad \text{for } j \in \N, \qquad 
   f^{(\infty)}_{\eps}(x, X) = |X|^p 
   \quad \forall\, \eps \] 
shows. 

However, again as a direct application of our closure results, we obtain that a stronger notion of equivalence does in fact imply commutability of these $\Gamma$-limits. This stronger condition in particular is satisfied if 
\[ \lim_{j \to \infty} \sup_{\eps > 0} \int_{U} 
   \sup_{|X| \le R} | f^{(j)}_{\eps}(x, X) - f^{(\infty)}_{\eps}(x , X) | \, dx = 0 \]
and thus holds true in our applications to homogenization to be discussed below. In the following it suffices to consider a single domain $\Omega$. 

\begin{theo}[Commutability]\label{theo:commutability}
Let $\Omega \subset \R^n$ be bounded and open. Suppose that the family of functionals $F^{(j)}_{\eps}$, $ j \in \N \cup \{ \infty \}$, $ \eps > 0$, with densities $f^{(j)}_{\eps} : \Omega \times \M{m}{n}$ is of uniform $p$-G{\r a}rding type on $\Omega$. Assume that
\begin{itemize} 
\item[(i)] For each $j < \infty$ the $\Gamma$-limit $\Gamma(L^p)\text{-}\lim_{\eps \to 0} F^{(j)}_{\eps} =: F^{(j)}_0$ exists. 

\item[(ii)] The families $((f^{(j)}_{\eps})_{\eps > 0})_{j \in \N}$ and $(f^{(\infty)}_{\eps})_{\eps > 0}$ are equivalent on $\Omega$ and moreover 
\[ \lim_{j \to \infty} \int_{\Omega} 
   \sup_{|X| \le R} | f^{(j)}_{\eps}(x, X) - f^{(\infty)}_{\eps}(x , X) | \, dx = 0 \]
for every $R > 0$ and all $\eps > 0$.
\end{itemize}

Then $\Gamma(L^p)\text{-}\lim_{\eps \to 0} F^{(\infty)}_{\eps} = \Gamma(L^p)\text{-}\lim_{j \to \infty} F^{(j)}_{0} =: F^{(\infty)}_0$. Moreover, $\Gamma(L^p)\text{-}\lim_{j \to \infty} F^{(j)}_{\eps} = \lsc F^{(\infty)}_{\eps}$ and 
\[ \Gamma(L^p)\text{-}\lim_{\eps \to 0} \Big( \Gamma(L^p)\text{-}\lim_{j \to \infty} F^{(j)}_{\eps} \Big)
   = \Gamma(L^p)\text{-}\lim_{j \to \infty} \Big( \Gamma(L^p)\text{-}\lim_{\eps \to 0} F^{(j)}_{\eps} \Big), \]
i.e., the following diagram commutes: 
\begin{figure}[h!]
\null\hspace{5.3cm}
\begindc{\commdiag}[500]
\obj(0,2)[aa]{$ F^{ (j) }_{\eps} $}
\obj(2,2)[bb]{$ \lsc F^{ (\infty) }_{\eps} $}
\obj(0,0)[cc]{$ F^{ (j) }_{0} $}
\obj(2,0)[dd]{$ F^{ (\infty) }_{0} $}
\mor{aa}{bb}{$ \Gamma $}
\mor{aa}{cc}{$ \Gamma $}
\mor{bb}{dd}{$ \Gamma $}
\mor{cc}{dd}{$ \Gamma $}
\enddc
\end{figure}
\end{theo}

\begin{proof}
The first assertion follows immediately from Theorem \ref{theo:Gamma-closure-single-Garding}. Next, Theorem \ref{theo:relaxation} applied to $F^{(j)}_{\eps}$ with fixed $\eps$ gives $\Gamma(L^p)\text{-}\lim_{j \to \infty} F^{(j)}_{\eps} = \lsc F^{(\infty)}_{\eps}$. It remains to note that $\Gamma(L^p)\text{-}\lim_{\eps \to 0} F^{(\infty)}_{\eps} = \Gamma(L^p)\text{-}\lim_{\eps \to 0} \lsc F^{(\infty)}_{\eps}$. 
\end{proof}

\begin{figure}[h!]
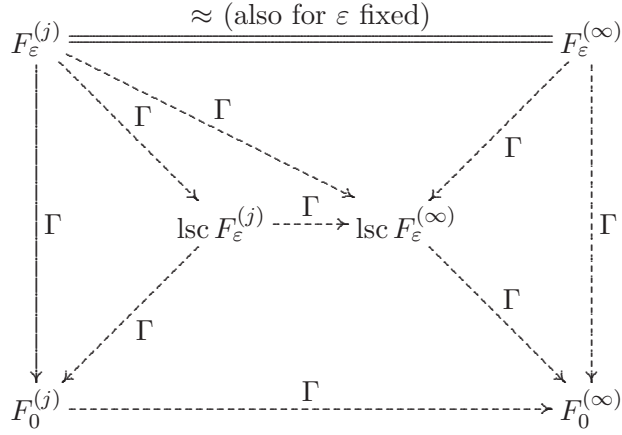

\null\hspace{3.3cm}
\begindc{\commdiag}[700]
\obj(0,2)[aa]{$ F^{ (j) }_{\eps} $}
\obj(3,2)[bb]{$ F^{ (\infty) }_{\eps} $}
\obj(1,1)[cc]{$ \lsc F^{ (j) }_{\eps} $}
\obj(2,1)[dd]{$ \lsc F^{ (\infty) }_{\eps} $}
\obj(0,0)[ee]{$ F^{ (j) }_{0} $}
\obj(3,0)[ff]{$ F^{ (\infty) }_{0} $}
\mor{aa}{bb}{$ \approx $ (also for $ \eps $ fixed)}[\atleft,\equalline]
\mor{aa}{cc}{$ \Gamma $}[\atleft,1]
\mor{bb}{dd}{$ \Gamma $}[\atleft,1]
\mor{cc}{dd}{$ \Gamma $}[\atleft,1]
\mor{aa}{dd}{$ \Gamma $}[\atleft,1]
\mor{cc}{ee}{$ \Gamma $}[\atleft,1]
\mor{dd}{ff}{$ \Gamma $}[\atleft,1]
\mor{aa}{ee}{$ \Gamma $}
\mor{bb}{ff}{$ \Gamma $}[\atleft,1]
\mor{ee}{ff}{$ \Gamma $}[\atleft,1]
\enddc\caption{Schematic diagram of the assumptions (solid lines) and consequences (dashed) in Theorem \ref{theo:commutability}. The arrows indicate $\Gamma$-convergence with respect to $j$, $\eps$ or of the constant sequence with fixed indices (the relaxation). }
\end{figure}

In fact, under the condition 
\[ \lim_{j \to \infty} \sup_{\eps > 0} \int_{U} 
   \sup_{|X| \le R} | f^{(j)}_{\eps}(x, X) - f^{(\infty)}_{\eps}(x , X) | \, dx = 0 \]
we also have $ \Gamma $-convergence along any diagonal sequence $(\eps_k, j_k)_{k \in \N}$. 

\begin{theo}[Simultaneous limits]\label{theo:simultaneous-limits}
Let $\Omega \subset \R^n$ be bounded and open. Suppose that the family of functionals $F^{(j)}_{\eps}$, $ j \in \N \cup \{ \infty \}$, $ \eps > 0$, with densities $f^{(j)}_{\eps} : \Omega \times \M{m}{n}$ is of uniform $p$-G{\r a}rding type on $\Omega$. Let $(j_k)_{k \in \N}$ and $(\eps_k)_{k \in \N}$ be subsequences of $j$ and $\eps$, respectively. Assume that $\Gamma(L^p)\text{-}\lim_{\eps \to 0} F^{(\infty)}_{\eps} = F^{(\infty)}_0$ and 
\[ \lim_{k \to \infty} \int_{\Omega} 
   \sup_{|X| \le R} | f^{(j_k)}_{\eps_k}(x, X) - f^{(\infty)}_{\eps_k}(x , X) | \, dx = 0 \]
for every $R > 0$ and all $\eps > 0$. Then $\Gamma(L^p)\text{-}\lim_{k \to \infty} F^{(j_k)}_{\eps_k} = F^{(\infty)}_0$. 
\end{theo}

\begin{proof}
This is nothing but Theorem \ref{theo:perturbation} applied to $F^{(j_k)}_{\eps_k}$ and $F^{(\infty)}_{\eps_k}$. 
\end{proof}

Note that $\Gamma(L^p)\text{-}\lim_{\eps \to 0} F^{(\infty)}_{\eps} = F^{(\infty)}_0$ is known to exist under the assumptions of Theorem \ref{theo:Gamma-closure-variable-Garding}.

\subsection{Homogenization of G{\r a}rding type functionals}

As a corollary to our $\Gamma$-closure and commutability theorems we first obtain a homogenization closure theorem which generalizes the corresponding result of Braides to integral densities which are not assumed to be quasiconvex in their second argument and which instead of a standard $p$-growth assumption are only assumed to induce G{\r a}rding type integral functionals. Secondly we also obtain a general criterion for the interchangeability of homogenizing and taking the $\Gamma$-limit of a sequence of functionals. In view of our applications in the following sections we state the closure theorem on variable domains and with densities bounded below by a singled density inducing a G{\r a}rding inequality. Specific situations in which the assumptions of the following results are known to be satisfied will be discussed in a later section. 

For Borel functions $f^{(\infty)}, f^{(1)}, f^{(2)}, \ldots : \R^n \times \M{m}{n} \to \R$ and all $\eps > 0$ we consider the integral functionals $F^{(j)}_{\eps} : L^p(\R^n; \R^m) \times {\cal A}(\R^n) \to \R \cup \{+ \infty\}$ which for $u \in W^{1, p}(\R^n; \R^m)$ and $U \in {\cal A}(\R^n)$ are given by 
\begin{align*}
  F^{(j)}_{\eps}(u, U) 
  = \int_U f^{(j)} \Big( \frac{x}{\eps}, \nabla u(x) \Big) \, dx 
\end{align*}
and take the value $+ \infty$ otherwise. For the Borel functions $f^{(\infty)}_{\rm hom}, f^{(1)}_{\rm hom}, f^{(2)}_{\rm hom}, \ldots : \M{m}{n} \to \R$ we define $F^{(j)}_{\rm hom} : L^p(\R^n; \R^m) \times {\cal A}(\R^n) \to \R \cup \{+ \infty\}$ by 
\begin{align*}
  F^{(j)}_{\rm hom}(u, U) 
  = \int_U f^{(j)}_{\rm hom} (\nabla u(x)) \, dx 
\end{align*}
for $u \in W^{1, p}(\R^n; \R^m)$, $U \in {\cal A}(\R^n)$ and $+ \infty$ otherwise.

\begin{theo}[Homogenization closure]
\label{theo:homogenization-closure}
Let us for each $ j \in \N \cup \{ \i \} $ have a Borel function $ f^{(j)}: \R^n \times \M{m}{n} \to \R $ 
such that there are a constant $\beta > 0$ and a Borel function $g: \M{m}{n} \to \R$ with 
\[ g(X) \le f^{(j)}(x,X) \le \beta( |X|^{p} + 1 ) \] 
for all $x \in \R$, $X \in \M{m}{n}$ and $j \in \N \cup \{\infty\}$. Assume that for all $U \in {\cal A}(\R^n)$ there are $\alpha_U > 0, \gamma_U$ such that 
\[ \int_{U} g(\nabla u(x)) \, dx 
   \ge \alpha_U \int_{U} |\nabla u(x)|^p \, dx - \gamma_U \int_{U} |u(x)|^p \, dx \]
for $u \in W^{1, p}(U; \R^m)$. Suppose that 
\begin{itemize}
\item[(i)] for every $j \in \N$ the function $ f^{(j)} $ is homogenizable, i.e., there exists a Borel function $f^{(j)}_{\rm hom} : \M{m}{n} \to \R $ such that
\[ \Gamma(L^p)\text{-}\lim_{\eps \to 0} F^{(j)}_{\eps}(\cdot, U) 
   = F^{(j)}_{\rm hom}(\cdot, U) \]
for every $ U \in { \cal A }( \R^{n} ) $ and 
\item[(ii)] for every $R > 0$ 
\[ \lim_{j \to \infty} \limsup_{T \to \infty} \dashint_{(-T, T)^n} \sup_{|X| \le R} 
   |f^{(j)}(x, X) - f^{(\infty)}(x, X)| \ dx 
   = 0. \]
\end{itemize}

Then $ f^{(\infty)} $ is also homogenizable, say $\Gamma(L^p)\text{-}\lim_{\eps \to 0} F^{(\infty)}_{\eps}(\cdot, U) = F^{(\infty)}_{\rm hom}(\cdot, U)$ for $ U \in { \cal A }( \R^{n} ) $. The limiting density $f^{(\infty)}_{\rm hom}$ is given by 
\begin{align*}
  f^{(\infty)}_{\rm hom}(X) 
  = \lim_{k \to \infty} \inf \bigg\{ \dashint_{(0, k)^n} f^{(\infty)}(x, X + \nabla u(x)) \ dx : u \in W^{1, p}_0((0, k)^n; \R^m) \bigg\} 
\end{align*}
and satisfies $f^{(\infty)}_{\rm hom}(X) = \lim_{j \to \infty} f^{(j)}_{\rm hom}(X)$ for every $X \in \M{m}{n}$. 
\end{theo}

Except for the representation formula of $f^{(\infty)}_{\rm hom}$ this is a direct consequence of Theorem \ref{theo:Gamma-closure-variable-Garding}. The representation formula is implied by the following proposition which in turn is a consequence of Theorem \ref{theo:Gamma-closure-boundary-values} and Proposition \ref{prop:coercivity-small-domains}. 

\begin{prop}[Representation formula]
\label{prop:representation-formula}
Suppose $f: \R^n \times \M{m}{n} \to \R$ is a Borel function satisfying the growth assumptions with $\beta$ and $g$ as in Theorem \ref{theo:homogenization-closure}. If there exists $f_{\rm hom} : \M{m}{n} \to \R$ such that 
\[ \Gamma(L^p)\text{-}\lim_{\eps \to 0} F_{\eps}(\cdot, U) 
   = F_{\rm hom}(\cdot, U) \] 
for all $U \in {\cal A}(\R^{n})$, then 
\begin{align*}
  f_{\rm hom}(X) 
  = \lim_{k \to \infty} \inf \bigg\{ \dashint_{ k U } f(x, X + \nabla \varphi(x)) \ dx : \varphi \in W^{1, p}_0( k U ; \R^m ) \bigg\} 
\end{align*}
for every $X \in \M{m}{n}$, $U \in {\cal A}(\R^{n})$.
\end{prop}

\begin{proof}
Fix any $U \in {\cal A}(\R^n)$. Let $ \Omega $ be an open ball centered in 0 such that $ U \subset \Omega $. 
Take $X \in \M{m}{n}$ and set $u_0(x) = \ell_X(x) = Xx$. By Proposition \ref{prop:coercivity-small-domains} the functionals $\bar{F}_{\eps}(\cdot, U')$ are then equicoercive on $W^{1, p}(U'; \R^m)$
for $ U' := \frac{1}{K} U $ with $ K > 0 $ sufficiently large. Quasiconvexity of $f_{\rm hom}$ and Theorem \ref{theo:Gamma-closure-boundary-values} thus imply that for $\eps_k = \frac{1}{kK}$
\begin{align*}
  f_{\rm hom}(X)
  &= \min \{ \tfrac{1}{|U'|}  F_{\rm hom}(u, U') : u \in \ell_X + W^{1, p}_0 (U' ; \R^m) \} \\  
  &= \min \{ \tfrac{1}{|U'|} \bar{F}_{\rm hom}(u, U') : u \in L^p(U'; \R^m) \} \\ 
  &= \lim_{k \to \infty} \inf \{ \tfrac{1}{|U'|} \bar{F}_{\eps_k}(u, U') : u \in L^p(U' ; \R^m) \} \\ 
  &= \lim_{k \to \infty} \inf \bigg\{ \dashint_{U'} f (k K x, \nabla u(x)) \, dx : u \in \ell_X + W^{1,p}_{0}(U'; \R^m) \bigg \} \\ 
  &= \lim_{k \to \infty} \inf \bigg\{ \dashint_{kU} f (x, X + \nabla \varphi(x)) \, dx : \varphi \in W^{1,p}_{0}(kU; \R^m) \bigg \}. \qedhere
\end{align*}
\end{proof}

\begin{proof}[Proof of Theorem \ref{theo:homogenization-closure}]
Let 
\begin{align*}
  f^{(j)}_{\eps}(x, X) = f^{(j)} \Big( \frac{x}{\eps}, X \Big) 
\end{align*}
for $j \in \N \cup \{ \infty \}$ and $\eps > 0$ and set 
\begin{align*}
  f^{(j)}_{0}(x, X) := f^{(j)}_{\rm hom} (X) 
\end{align*}
for $j \in \N$. For $U \in {\cal A}(\R^n)$ with $U \subset (-T_0, T_0)^n$ and $R > 0$ we then have 
\begin{align*}
  &\limsup_{\eps \to 0} \int_U \sup_{|X| \le R} \big| f^{(j)}_{\eps}(x, X) - f^{(\infty)}_{\eps}(x , X) \big| \, dx \\ 
  &\le \limsup_{\eps \to 0} \int_{(-T_0, T_0)^n} \sup_{|X| \le R} \Big| f^{(j)} \Big(\frac{x}{\eps}, X \Big) - f^{(\infty)} \Big( \frac{x}{\eps}, X \Big) \Big| \, dx \\ 
  &= \limsup_{\eps \to 0} \eps^n \int_{(-T_0/\eps, T_0/\eps)^n} \sup_{|X| \le R} \big| f^{(j)} (x, X) - f^{(\infty)} (x, X) \big| \, dx \\ 
  &= (2T_0)^{n} \limsup_{T \to \infty} \dashint_{(-T, T)^n} \sup_{|X| \le R} \big| f^{(j)} (x, X) - f^{(\infty)} (x, X) \big| \, dx \\ 
  &= 0.
\end{align*}

Thus we may apply Theorem \ref{theo:Gamma-closure-variable-Garding} and deduce that there exists a Borel function $ f^{(\infty)}_{0} : \R^n \times \M{m}{n} \to \R $ and corresponding integral functional $F^{(\infty)}_{0}$ such that 
\begin{align*}
  \Gamma(L^p)\text{-}\lim_{\eps \to 0} F^{(\infty)}_{\eps}(\cdot, U)
   = F^{(\infty)}_0(\cdot, U)
\end{align*}
for all $U \in {\cal A}(\R^n)$ and moreover 
\[ f^{(j)}_0 (\cdot, X) \weakstar f^{(\infty)}_0(\cdot, X) \] 
in $L^{\infty}(\R^n)$ for all $X \in \M{m}{n}$. Since $f^{(j)}_{0}(x, X) = f^{(j)}_{\rm hom} (X)$ is independent of $x$ we in fact have that $f^{(\infty)}_{0}(x, X) = f^{(\infty)}_{\rm hom}(X)$ for some Borel function $f^{(\infty)}_{\rm hom}$ for a.e.\ $x \in \R^n$ and all $X \in \M{m}{n}$. (As a quasiconvex function $f^{(\infty)}_{\rm hom}$ is even continuous.) In particular, 
\[ f^{(j)}_{\rm hom} (X) \to f^{(\infty)}_{\rm hom}(X) \] 
for all $X \in \M{m}{n}$. 
\end{proof}

We finally give an equivalence condition adapted to homogenization leading to a $\Gamma$-commuting diagram as in Theorem \ref{theo:commutability}. The analogous adaption of Theorem \ref{theo:simultaneous-limits} is straightforward.

\begin{theo}
\label{theo:homogenization-commutability}
Suppose $(f^{(j)})_{j \in \N}$ is a sequence of Borel functions satisfying the assumptions of Theorem \ref{theo:homogenization-closure} and in addition 
\[ \lim_{j \to \infty} \int_{(-T, T)^n} \sup_{|X| \le R} 
   |f^{(j)}(x, X) - f^{(\infty)}(x, X)| \ dx 
   = 0 \]
for every $ T \in \N $ and $ X \in \M{m}{n} $. 

Then the limit $ f^{(\infty)}( x , \cdot ) $ is also homogenizable and the following diagram commutes for every $ U \in {\cal A}( \R^{n} ) $:
\begin{figure}[h!]
\null\hspace{5.3cm}
\begindc{\commdiag}[500]
\obj(0,2)[aa]{$ F^{ (j) }_{ \eps }( \cdot , U ) $}
\obj(2,2)[bb]{$ \lsc F^{ (\i) }_{ \eps }( \cdot , U ) $}
\obj(0,0)[cc]{$ F^{ (j) }_{\hom}( \cdot , U ) $}
\obj(2,0)[dd]{$ F^{(\i)}_{\hom}( \cdot , U ) $}
\mor{aa}{bb}{$ \Gamma $}[\atleft,1]
\mor{aa}{cc}{$ \Gamma $}
\mor{bb}{dd}{$ \Gamma $}[\atleft,1]
\mor{cc}{dd}{$ \Gamma $}[\atleft,1]
\enddc
\end{figure}
\end{theo}

\begin{remark}
Note that if $f^{(j)}( x , \cdot ) \to f^{(\infty)}( x , \cdot )$ locally uniformly for almost every $ x \in \R^{n} $, then indeed 
\[ \lim_{j \to \infty} \int_{(-T, T)^n} \sup_{|X| \le R} 
   |f^{(j)}(x, X) - f^{(\infty)}(x, X)| \ dx 
   = 0 \]
for every $ T \in \N $ and $ X \in \M{m}{n} $. Vice versa, if this condition is satisfied, then there is a subsequence $ ( f^{(j_{k})} )_{k \in \N} $ such that $f^{(j_{k})}( x , \cdot ) \to f^{(\i)}( x , \cdot )$ locally uniformly almost everywhere. 
\end{remark}

\begin{proof}[Proof of Theorem \ref{theo:homogenization-commutability}.]
This is an immediate consequence of Theorem \ref{theo:commutability}. 
\end{proof}

\section{Applications in elasticity theory}\label{section:Elasticity}

In this chapter we discuss some applications of our abstract results to elasticity theory. If $y \in W^{1, p}(\Omega; \R^n)$ is a deformation of a hyperelastic body with reference configuration $\Omega \subset \R^n$, a bounded domain with Lipschitz boundary, and stored energy function $W : \Omega \times \M{n}{n} \to \R$, then the elastic energy of $y$ is given as 
\[ E(y) 
   = \int_{\Omega} W(x, \nabla y(x)) \, dx. \] 
Our homogenization closure and commutability results now shed light on the effective behavior of elastic materials in the limiting regimes 1.\ of highly oscillatory material properties (measured in terms of a small parameter $\eps$) and 2.\ of small displacements (measured in terms of $\delta \ll 1$). As $\eps \to 0$, we are led to a homogenization problem in nonlinear elasticity as has been studied in \cite{Braides:85, Mueller}, while the limit $\delta \to 0$ corresponds to a linearization of the energy functional, cf.\ \cite{DalMasoNegriPercivale,Schmidt}. Here we will allow for an explicit dependence of the energy density on $\delta$. This will enable us to also study effective functionals for stored energy functions with multiple wells, separated by a distance of the order of the displacements. As our limiting energy functional in this case is still nonlinear, the limit $\delta \to 0$ actually corresponds to a geometric linearization only. 

\subsection{Functionals for microstructured multiwell materials}

We consider the family of functionals $({\cal I}_{\eps, \delta})_{ \eps , \delta > 0 } $ given
for $ u \in W^{1,p}( \Omega; \R^{n} ) $ by
\[ {\cal I}_{ \eps , \delta }(u) 
:= \frac{1}{ \delta^{p} } \int_{ \Omega } W_{ \delta } \Big( \frac{x}{ \eps } , I + \delta \nabla u(x) \Big) \ dx , \]
and extended to $ L^{p}( \Omega; \R^{n} ) $ by $ +\infty $, the physically relevant cases being $n = 1,2,3$ and $p = 2$. Henceforth we will assume that $W_{\delta} : \R^n \times \M{n}{n} \to \R$ satisfies the physically reasonable objectivity condition $W_{\delta}(x, RX) = W_{\delta}(x, X)$ for all $x \in \Omega$, $X \in \M{n}{n}$ and $R \in {\rm SO}(n)$. 

For fixed $\delta$ the limit $\eps \to 0$ then leads to a classical homogenization problem rendering homogenized densities $W^{\rm hom}_{\delta}$ under suitable assumptions. 
For the limit $\delta \to 0$ with fixed $\eps$ let us first consider the standard single well case. Suppose that $I$ is (up to a rigid motion) the unique stress free strain and that $W_{\delta} = W$, not explicitly depending on $\delta$, has a single well structure. More precisely, let
\[ W(x, X) = 0 \iff X \in {\rm SO}(n)
   \quad \mbox{and} \quad
   W(x, X)
   \ge c \, \dist^p(X, {\rm SO}(n)) \quad \forall \, X \in \M{n}{n} \]
for a $c > 0$. For $p = 2$, sending $\delta \to 0$ under suitable regularity assumption leads to $\delta^{-2} W(\frac{x}{\eps}, I + \delta X) \to \frac{1}{2} D^2_X W(\frac{x}{\eps}, I)[X, X]$ and a corresponding energy functional in linear elasticity. By frame indifference, this density in fact only depends on the symmetric part $e(u) = X_{\sm} = \frac{X^T + X}{2}$ of the displacement gradient $\nabla u = X$. 

Here we would like to more generally also allow for multiwell energies whose minimizers are concentrated in a $\delta$-neighborhood of ${\rm SO}(n)$. We therefore impose a slightly milder non-degeneracy and growth condition in addition to our requirement of frame indifference: 
\begin{definition}
A family $W_{\delta}: \R^n \times \M{n}{n} \to \R$, $\delta > 0$, of Borel functions is said to be an admissible family of nonlinearly elastic stored energy functions if for a.e.\ $x \in \R^n$, all $X \in \M{n}{n}$ and $R \in {\rm SO}(n)$ and suitable $c, C > 0$ 
\begin{itemize}
\item[(i)] $W_{\delta}(x, RX) = W_{\delta}(x, X)$ and  
\item[(ii)] $c \, \dist^p(X, {\rm SO}(n)) - C \delta^p \le W_{\delta}(x, X) \le C(|X|^p + 1)$. 
\end{itemize}
\end{definition}

\begin{example}
In particular, this assumption is compatible with models for shape memory alloys in their martensitic phase, when the zero-set of $W_{\delta} \ge 0$ consists of multiple wells and is of the form 
\[ W_{\delta}(x, X) = 0 \iff X \in \bigcup_{i = 1}^N {\rm SO}(n) ( I + \delta U_i ) \] 
for given positive matrices $U_i \in \M{n}{n}_{\sm}$. It is not hard to see that our non-degeneracy assumption is satisfied if we require that $W_{\delta}(x, X) \ge c\, \dist^p\big(X, \bigcup_{i = 1}^N {\rm SO}(n) ( I + \delta U_i )\big)$. 
\end{example}

By way of contrast to the single well case, for multiwell energies $W_{\delta}$ one may only assume that $\delta^{-p} W_{\delta}(x; I + \delta X) \to V(x, X)$ for $X \in \M{n}{n}_{\sm}$, where $V$ is a linearly frame indifferent and non-degenerate multiwell energy density acting on symmetrized displacement gradients inducing a nonlinear limiting functional in (geometrically) linearized elasticity. Accordingly we define 
\[ {\cal I}_{\eps,0} (u) 
:= \int_{ \Omega } V \Big( \frac{x}{ \eps } , e(u)(x) \Big) \ dx , \]
for $u \in W^{1, p}(\Omega; \R^n)$ and extended to $ L^{p}( \Omega , \R^{n} ) $ by $ + \infty $ and summarize our admissibility assumptions on linear non-degeneracy, frame indifference and growth in the following definition.
\begin{definition}
A Borel function $V : \R^n \times \M{n}{n} \to \R$ is an admissible linearly elastic stored energy function if for a.e.\ $x \in \R^n$, all $X \in \M{n}{n}$ and suitable $c, C > 0$ 
\begin{itemize}
\item[(i)] $V(x, X) = V(x, X_{\sm})$ and 
\item[(ii)] $c |X_{\sm}|^p \le V(x, X) \le C(|X|^p + 1)$. 
\end{itemize}
\end{definition}

Whereas in the single well case the limiting density $\frac{1}{2} D^2_X W(\frac{x}{\eps}, I)[X, X]$ is quadratic and thus convex in $X$, the multiwell density $V$ in general will not even be quasiconvex. (For compatible matrices $U_1, \ldots, U_N$ in the above example the quasiconvex envelope on linear strains is strictly larger than $\{U_1, \ldots, U_N\}$.) However, since in solving variational problems for ${\cal I}_{\eps, 0}$ one may pass to its relaxation $\lsc {\cal I}_{\eps, 0}$, in order to justify that ${\cal I}_{\eps,0}$ is an adequate description of variational problems in the small displacement regime, we will therefore ask if $\Gamma(L^p)\text{-}\lim_{\delta \to 0} {\cal I}_{\eps, \delta} = \lsc {\cal I}_{\eps, 0}$. (Also cf.\ the discussion in \cite{Schmidt} for the $\eps$-independent case).  

Korn's inequality in Theorem \ref{theo:Korn} and its nonlinear counterpart Theorem \ref{theo:Korn-nonlinear}, which is based on the geometric rigidity estimate Theorem \ref{theo:geometric-rigidity} from \cite{FrieseckeJamesMueller}, provide the link to our theory developed for G{\r a}rding type functionals: 
\begin{lemma}\label{lemma:rigity-Garding}
Suppose that the $W_{\delta}, V : \R^n \times \M{n}{n} \to \R$ are admissible (non-)linearly elastic stored energy functions. Then the family of functionals ${\cal I}_{\eps, \delta}$ with $\eps > 0$ and $\delta \ge 0$ is of uniform $p$-G{\r a}rding type on every $U \in {\cal A}(\R^n)$. 
\end{lemma}

\begin{proof}
For $\delta = 0$ this is immediate from Theorem \ref{theo:Korn}. For $\delta > 0$ it follows from Theorem \ref{theo:Korn-nonlinear} applied to $\delta u$. 
\end{proof}

\subsection{Homogenization and geometric linearization}

Our main result is the following theorem on homogenization and geometric linearization for multiwell energy functionals. In particular we will see that these limiting processes commute and may be taken simultaneously. As it turns out, the (geometric) linearization and commutability results in \cite{DalMasoNegriPercivale,Schmidt,MuellerNeukamm,GloriaNeukamm} are direct consequences that theorem. 

\begin{theo}[Homogenization closure and commutability]
\label{theo:multiwells}
Suppose that the $W_{\delta}, V : \R^n \times \M{n}{n} \to \R$ are admissible (non-)linearly elastic stored energy functions. Assume that 
\begin{itemize}
\item[(i)] for every $\delta > 0$ the function $W_{\delta}$ is homogenizable, i.e., there exists a Borel function $W_{\delta}^{\rm hom} : \M{n}{n} \to \R $ such that
\[ \Gamma(L^p)\text{-}\lim_{\eps \to 0} {\cal I}_{\eps, \delta}(\cdot, U) 
   = {\cal I}_{\delta}^{\rm hom}(\cdot, U) \]
for every $ U \in { \cal A }( \R^{n} ) $ and 
\item[(ii)] for every $R > 0$ 
\[ \lim_{\delta \to 0} \sup_{T > 0} \dashint_{(-T, T)^n} \sup_{X \in \M{n}{n}_{\sm} \atop |X| \le R} 
   |\delta^{-p} W_{\delta}(x, I + \delta X) - V(x, X)| \ dx 
   = 0 \]
and 
\[ \lim_{|Z| \to 0} \sup_{T > 0} \dashint_{(-T, T)^n} \sup_{|X| \le R} |V(x, X + Z) - V(x, X)| \, dx 
   = 0. \]
\end{itemize}

Then also $V$ is homogenizable, say $\Gamma(L^p)\text{-}\lim_{\eps \to 0} {\cal I}_{\eps,0}(\cdot, U) =: {\cal I}_0^{\rm hom}(\cdot, U)$ for $ U \in { \cal A }( \R^{n} ) $. The limiting density $V^{\rm hom}$ is given by 
\begin{align*}
  V^{\rm hom}(X) 
  = \lim_{k \to \infty} \inf \bigg\{ \dashint_{(0, k)^n} V(x, X + e(u)(x)) \ dx : u \in W^{1, p}_0((0, k)^n; \R^n) \bigg\} 
\end{align*}
and satisfies $V^{\rm hom}(X) = \lim_{\delta \to 0} \delta^{-p} W_{\delta}^{\rm hom}(I + \delta X)$ for every $X \in \M{n}{n}$. Moreover, also every simultaneous limit $\eps_k, \delta_k \to 0$ satisfies 
\[ \Gamma(L^p)\text{-}\lim_{k \to \infty} {\cal I}_{\eps_k,\delta_k} 
   = {\cal I}_0^{\rm hom} \] 
and the following diagram commutes for every $ U \in {\cal A}( \R^{n} ) $: 
\begin{figure}[h!]
\null\hspace{5.3cm}
\begindc{\commdiag}[500]
\obj(0,2)[aa]{$ {\cal I}_{\eps,\delta}(\cdot, U) $}
\obj(2,2)[bb]{$ \lsc {\cal I}_{\eps, 0}(\cdot , U) $}
\obj(0,0)[cc]{$ {\cal I}_{\delta}^{\rm hom}(\cdot, U) $}
\obj(2,0)[dd]{$ {\cal I}_0^{\rm hom}(\cdot, U) $}
\mor{aa}{bb}{$ \Gamma $}[\atleft,1]
\mor{aa}{cc}{$ \Gamma $}
\mor{bb}{dd}{$ \Gamma $}[\atleft,1]
\mor{cc}{dd}{$ \Gamma $}[\atleft,1]
\enddc
\end{figure}\\
Here 
\[ \lsc {\cal I}_{\eps, 0}(u, U) 
   = \int_U V^{\qc}\Big( \frac{x}{\eps}, e(u)(x) \Big) \, dx \]
for $u \in W^{1, p}(\Omega; \R^n)$ and $+ \infty$ otherwise, where $V^{\qc}$ denotes the quasiconvex envelope of $V$ in its second argument. 
\end{theo}

\begin{remark}\label{remark:muwe}
\begin{enumerate}
\item Assumption (i) is known to be satisfied in two cases of special interest: periodic and stochastic homogenization. More precisely, if the $W_{\delta}$ are periodic in $x$, then the results on periodic homogenization of Braides \cite{Braides:85} and M{\"u}ller \cite{Mueller} directly apply to ${\cal I}_{\eps, \delta}$ for fixed $\delta$. Their results have been extended to the field of stochastic homogenization by Messaoudi and Michaille in \cite{MessaoudiMichaille} (based on the ideas in \cite{DalMasoModica}): Suppose that for every $ \xi $ from some probability space we have an admissible nonlinearly elastic stored energy function $ W_{\delta}^{\xi} $ such that the corresponding mapping $ \xi \mapsto {\cal I}_{1,\delta}^{\xi}(\cdot, U) $ is a random integral functional, periodic in law and ergodic. Then the functions $ W_{\delta}^{\xi} $ are uniformly homogenizable, i.e., for almost every $ \xi $ there exists a homogenized density $ W_{\delta}^{\hom} $, which moreover does not depend on $ \xi $. Therefore, Assumption (i) is satisfied for almost every $\xi$ along any null sequence $\delta \to 0$. (If $W_{\delta}^{\xi} = W^{\xi}$ does not explicitly depend on $\delta$, Assumption (i) is satisfied for every $\delta$ for almost every $\xi$.) We however remark that in \cite{MessaoudiMichaille} the authors assume the local Lipschitz condition
\begin{align}\label{Wdo-loc-Lip}
  |W_{\delta}^{\xi}(x, X) - W_{\delta}^{\xi}(x, X')|
  \le C(1 + |X|^{p-1} + |X'|^{p-1}) |X - X'|,
\end{align}
which may in fact be omitted, see \cite{Jesenko}. 

\item The second assumption in (ii) is a mild continuity assumption on $V$. If, e.g., $V$ is a Carath{\'e}odory function which is periodic in $x$, this condition automatically holds by the Scorza-Dragoni theorem. It is also easily seen to be satisfied if the first assumption in (ii) is true and the $W_{\delta}$ fulfill \eqref{Wdo-loc-Lip}. 

The equivalence condition in (ii) only requires that $\delta^{-p} W_{\delta}(x, I + \delta X)$ be close to $V(x, X)$ when $X$ is symmetric. This assumption is satisfied in particular if $\delta^{-p} W_{\delta}(x, I + \delta X) \to V(x, X)$ as $\delta \to 0$ uniformly in $x \in \R^n$ and uniformly in $X$ on compact subsets of $\M{n}{n}_{\sm}$. In particular this is true if $W_{\delta} = W$ admits a quadratic Taylor expansion as in \cite{DalMasoNegriPercivale,MuellerNeukamm,GloriaNeukamm}.

\item Combining the previous two remarks (and Remark \ref{remark:no-upper}) we observe that the (geometric) linearization and commutability results in \cite{DalMasoNegriPercivale,Schmidt,MuellerNeukamm,GloriaNeukamm} are implied by Theorem \ref{theo:multiwells}. 

\item A straightforward adaption of Theorem \ref{theo:Gamma-closure-boundary-values} yields a version of Theorem \ref{theo:multiwells} for the functionals $\bar{\cal I}_{\eps, \delta}$ with prescribed boundary values. 

\item\label{rem:boundary-values-compactness} The proof of Theorem \ref{theo:Gamma-closure-boundary-values} shows that in fact also the homogenized functionals ${\cal I}_{\delta}^{\rm hom}$ are of $p$-G{\r a}rding type for all $\delta \ge 0$. It is worth noting that Theorem \ref{theo:p-Zhang} implies that in fact the homogenized densities also satisfy the non-degeneracy assumption 
\[ W_{\delta}^{\rm hom}(X) 
   \ge c \, \dist^p(X, {\rm SO}(n)) - C \delta^p, 
   \quad\text{respectively}, \quad 
   V^{\rm hom}(X) 
   \ge |X_{\sm}|. \]
In combination with \cite[Proposition 3.4]{DalMasoNegriPercivale} this observation implies equicoercivity of the functionals ${\cal I}_{\delta}^{\rm hom}$ under prescribed boundary conditions, thus complementing our $\Gamma$-convergence result. 
\end{enumerate}
\end{remark}

\begin{proof}
Korn's inequality (see Theorem \ref{theo:Korn}) and Theorem \ref{theo:lsc} show that $\lsc {\cal I}_{\eps, 0}$ is an integral functional with density $V^{\qc}$. By Lemma \ref{lemma:rigity-Garding} the remaining assertions of this theorem are a direct consequence of Theorems \ref{theo:homogenization-closure}, \ref{theo:homogenization-commutability} and \ref{theo:simultaneous-limits} (adapted to homogenization) once we have established that in fact 
\[ \lim_{\delta \to 0} \sup_{T > 0} \dashint_{(-T, T)^n} \sup_{X \in \M{n}{n} \atop |X| \le R} 
   |\delta^{-p} W_{\delta}(x, I + \delta X) - V(x, X)| \ dx 
   = 0 \]
for any $R > 0$, where the supremum inside the integral is taken over bounded subsets in $\M{n}{n}$ rather than $\M{n}{n}_{\sm}$. To this end, we decompose $I + \delta X \in \M{n}{n}$ as $I + \delta X = Q Y$ where $ Q \in SO(n) $ and $ Y \in \M{n}{n}_{\sm} $ in such a way that $Y = \sqrt{ (I + \delta X)^{T} (I + \delta X) }$ for $\det (I + \delta X) > 0$. By frame indifference 
\[ \tfrac{1}{ \delta^{p} } W_{ \delta }( x , I + \delta X ) 
   = \tfrac{1}{ \delta^{p} } W_{ \delta }( x , I + \delta X_{\sm} + q(\delta X) ) \]
for $q(\delta X) = Y - I + \delta X_{\sm}$. Note that $|Y| = |I + \delta X|$ grows at most linearly with $|\delta X|$ for large $|\delta X|$ whereas for small $|\delta X|$ a Taylor expansion gives $|Y - I - \delta X_{\sm}| \le C |\delta X|^2$, so that 
\[ |q(\delta X)| \le C \min \{|\delta X|, |\delta X|^2\}. \] 
Now 
\begin{align*}
  \tfrac{1}{\delta^{p}} W_{\delta}(x, I + \delta X) - V(x, X) 
  &= \tfrac{1}{\delta^{p}} W_{\delta}(x, I + \delta X_{\sm} + q(\delta X)) 
    - V(x, X_{\sm} + \delta^{-1} q(\delta X)) \\ 
  &\qquad + V(x, X_{\sm} + \delta^{-1} q(\delta X)) - V(x, X_{\sm}),   
\end{align*}
where by assumption 
\[ \lim_{\delta \to 0} \sup_{T > 0} \dashint_{(-T, T)^n} \sup_{X \in \M{n}{n} \atop |X| \le R} 
   |\delta^{-p} W_{\delta}(x, I + \delta X_{\sm} + q(\delta X)) - V(x, X_{\sm} + \delta^{-1} q(\delta X))| \, dx 
   = 0 \]
and also 
\[ \lim_{\delta \to 0} \sup_{T > 0} \dashint_{(-T, T)^n} \sup_{X \in \M{n}{n} \atop |X| \le R} 
   |V(x, X_{\sm} + \delta^{-1} q(\delta X)) - V(x, X_{\sm})| \, dx 
   = 0 \]
for any $R > 0$. 
\end{proof}

\begin{appendix}

\section{Appendix}

\subsection{$\Gamma$-convergence}

For the convenience of the reader we recall the definition and some important properties of $\Gamma$-convergence that were used frequently in our proofs.

\begin{definition}
Let $ F_{j} : M \to [-\i,\i]$, $ j \in \N $ be a sequence of functionals on a metric space $ (M,d) $. 
For $ x \in M $ we define the $\Gamma$-$\liminf$ and $\Gamma$-$\limsup$ at $x$ as 
\begin{align*}
\Gamma(d)\text{-}\liminf_{j \to \i} F_{j}( x ) &:= \inf \Big\{ \liminf_{j \to \infty} F_{j}( x_{j} ) : x_{j} \to x \Big\}, \\
\Gamma(d)\text{-}\limsup_{j \to \i} F_{j}( x ) &:= \inf \Big\{ \limsup_{j \to \infty} F_{j}( x_{j} ) : x_{j} \to x \Big\}. 
\end{align*}
If these values are equal, then we call it the $ \Gamma $-limit of $ ( F_{j} )_{j \in \N} $ in $x$ and write 
$ \Gamma(d) \text{-} \lim_{ j \to \infty } F_{j}( x ) $. Equivalently, $F_{\infty}(x) = \Gamma(d) \text{-} \lim_{ j \to \infty } F_{j}( x ) $ if the following conditions are satisfied:
\begin{itemize}
\item[(i)] (lim inf-inequality) If $ x_{j} \to x $ in $M$ then
$ \liminf_{j \to \i} F_{j}( x_{j} ) \ge F_{\i}(x) $. 
\item[(ii)]  (recovery sequence) There exists a sequence $ x_{j} \to x $ in $M$ such that
$ \lim_{j \to \i} F_{j}( x_{j} ) = F_{\infty}(x) $.
\end{itemize}
We say that $ ( F_{j} )_{j \in \N} $ $ \Gamma $-converges to some functional $F_{\infty}$, if and only if it $ \Gamma $-converges to $F_{\infty}(x)$ at every $x \in M$. 
\end{definition}

This convergence possesses Urysohn property.
\begin{theo}[see, e.g., {\cite[Proposition 7.11]{BraidesDefranceschi}}]\label{theo:Urysohn}
Take $ \lambda \in [-\i,\i] $ and $ x \in M $. Then $ \lambda = \Gamma(d) \mbox{-} \lim_{ j \to \i } F_{j}( x ) $ if and only if
for every subsequence $ ( F_{ j_{k} } )_{k \in \N} $ there exists a further subsequence $ ( F_{ j_{k_{l}} } )_{l \in \N} $
such that $ \lambda = \Gamma(d) \mbox{-} \lim_{ l \to \i } F_{ j_{k_{l}} }( x ) $. 
\end{theo}

On a separable metric space every sequence of functionals always contains a subsequence that $ \Gamma $-converges (cf., e.g., {\cite[Proposition 7.9]{BraidesDefranceschi}}). For a suitable class of integral functionals the following stronger compactness result holds:

\begin{theo}[see, e.g., {\cite[Theorem 12.5]{BraidesDefranceschi}}]\label{theo:standard-growth-Gamma-compactness}
Let $ f_{\eps} : \Omega \times \M{m}{n} \to \R, $ $ \eps > 0 $, be a family of Borel functions which for some $ \alpha , \beta > 0 $
satisfy the estimate
\[ \alpha |X|^{p} - \beta \le f_{\eps}(x,X) \le \beta( 1 + |X|^{p} ) \]
for all $ x \in \Omega $ and $ X \in \M{m}{n} $. Define for $ U \in {\cal A}( \Omega ) $ and $ u \in W^{1,p}( \Omega ; \R^{m} ) $
\[ F_{\eps}(u,U) := \int_{U} f_{\eps}( x , \nabla u(x) ) \ dx \]
and extend the definition to $ L^{p}( \Omega ; \R^{m} ) $ by $ \i $. Then, for every subsequence $ \eps_{j} \searrow 0 $ there
exists a further subsequence $ ( \eps_{ j_{k} } )_{ k \in \N } $ and a Carath\'{e}odory function $ \phi : \Omega \times \M{m}{n} \to \R $
satisfying the same growth estimate as $ ( f_{\eps} )_{\eps} $ such that 
\[ F_{0}( u , U ) 
   := \int_{U} \phi( x , \nabla u(x) ) \, dx 
   = \Gamma( L^{p} ) \mbox{-} \lim_{ k \to \infty } F_{ \eps_{ j_{k} } }( u , U ) \]
for all $u \in L^{p}(\Omega ; \R^{m})$ and $U \in {\cal A}( \Omega )$.
\end{theo}

If we have a constant sequence, i.e., $ F_{j} = F $ for all $ j \in \N $, then $ \Gamma \text{-} \lim_{j} F_{j} = \lsc F $
where $ \lsc $ stands for lower semicontinuous envelope (in the metric $d$). In our case the envelope may be determined by
\begin{theo}[{\cite[Statement III.7]{AcerbiFusco}}]\label{theo:lsc}
Let $ f: \R^{n} \times \M{m}{n} \to \R $ be a Carath\'{e}odory function which
satisfies for some $ p \ge 1 $ and $ \beta > 0 $
\[ 0 \le f(x,X) \le \beta( 1 + |X|^{p} ) \]
for almost all $ x \in \R^{n} $ and all $ X \in \M{m}{n} $. Define 
\[ F( u , U ) := \int_{U} f( x , \nabla u(x) ) \, dx \]
for $ u \in W^{1,p}( \R^{n} ; \R^{m} ) $ and $ U \in {\cal A}( \R^{n} ) $.
Then the sequentially weakly lower semicontinuous envelope of $ F( \cdot , U ) $ for every $ U \in {\cal A}( \R^{n} ) $ is 
\[ \swlsc F( u , U ) = \int_{U} f^{\qc}( x , \nabla u(x) ) \, dx. \]
\end{theo}
In particular, if $F(\cdot, U)$ satisfies a $p$-G{\r a}rding inequality as in Definition \ref{definition:Garding}, then its $L^p$-lower semicontinuous envelope is $\swlsc F( \cdot , U )$ extended by $+ \infty$ to $L^p$.

\subsection{Auxiliary results}

For easy reference we state here a number of auxiliary results that were crucial in our analysis. 

\subsubsection*{Equiintegrable modifications}

The following equiintegrability result of Fonseca, M{\"u}ller and Pedregal is crucial in the proof of our main closure theorem \ref{theo:Gamma-closure-single}
\begin{lemma}[see {\cite[Lemma 1.2]{FonsecaMuellerPedregal}} and {\cite[Lemma 8.3]{Pedregal}}]\label{lemma:equi-int}
Let $\Omega \subset \R^n$ be an open, bounded set and let $ (u_i)_{ i \in \N } $ be a bounded sequence in $W^{1, p}(\Omega; \R^m)$, $1 < p < \infty$. 
There exists a subsequence $ (u_{i_k})_{ k \in \N } $ and and a sequence $ (v_k)_{ k \in \N } \subset W^{1, p}(\Omega; \R^m) $ such that 
\[ \lim_{k \to \infty} | \{ \nabla v_k \ne \nabla u_{i_k} \} \cup \{ v_k \ne u_{i_k} \} | = 0   \] 
and $( |\nabla v_k|^p )_{ k \in \N } $ is equi-integrable. Moreover, if $u_i \weakly u $ in $W^{1, p}(\Omega; \R^m)$, then the $v_k$ can be chosen in such a way that $v_k = u$ on $\partial \Omega$ and $v_k \weakly u$ in $W^{1, p}(\Omega; \R^m)$. 
\end{lemma} 

\subsubsection*{Geometric rigidity}

The following geometric rigidity theorem proved in \cite{FrieseckeJamesMueller} (and extended to general $p$ in \cite{ContiSchweizer}) is a key step in the application of our abstract results to elasticity theory. 
\begin{theo}\label{theo:geometric-rigidity}
Suppose that $\Omega \subset \R^n$ is a bounded domain with Lipschitz boundary, $1 < p < \infty$. Then there exists a constant $C > 0$ such that for all $u \in W^{1, p}(\Omega; \R^n)$ there is an $R \in {\rm SO}(n)$ with 
\[ \| \nabla u - R \|_{L^p(\Omega; \M{n}{n})} 
   \le C \| \dist(\nabla u, {\rm SO}(n)) \|_{L^p(\Omega; \M{n}{n})}. \] 
\end{theo} 

This theorem can be seen as a nonlinear variant of Korn's inequality, where instead of rotations $R$ and the distance from ${\rm SO}(n)$ one has an analogous estimate for the distance to the set $\M{n}{n}_{\skw}$ of skew symmetric matrices in terms of a single matrix in $\M{n}{n}_{\skw}$. Another variant of Korn's inequality states that 
\begin{theo}\label{theo:Korn}
Suppose that $\Omega \subset \R^n$ is a bounded domain with Lipschitz boundary, $1 < p < \infty$. Then there exists a constants $C > 0$ such that for all $u \in W^{1, p}(\Omega; \R^n)$ 
\[ \| \nabla u \|_{L^p(\Omega; \M{n}{n})} 
   \le C \big( \| e(u) \|_{L^p(\Omega; \M{n}{n})} 
      + \| u \|_{L^p(\Omega; \M{n}{n})} \big), \] 
where $e(u) = \frac{(\nabla u)^T + \nabla u}{2}$. 
\end{theo} 
(Note that $\dist(X, \M{n}{n}_{\skw}) = |X_{\sm}|$.) Similarly, there is also an analogous version of the geometric rigidity estimate. 
\begin{theo}\label{theo:Korn-nonlinear}
Let $ \Omega \subset \R^{n} $ be a bounded Lipschitz domain and $ 1 < p < \infty $. There exists a constant 
$ C $, depending only on $\Omega$ and $p$, such that for every $ u \in W^{1,p}( \Omega , \R^{n} ) $
\[ \| \nabla u \|_{ L^{p}( \Omega , \M{n}{n} ) } 
   \le C \big( \| \dist \big( I + \nabla u , SO(n) \big) \|_{ L^{p}( \Omega ) } 
   + \| u \|_{ L^{p}( \Omega , \R^{n} ) } \big). \]
\end{theo}
Although this Corollary to Theorem \ref{theo:geometric-rigidity} appears to be well-known, a proof is hard to find in the literature, therefore we include it here.

\begin{proof}
We will show this result by contradiction. To this end, let us suppose that there exists
such sequence $ ( u_{j} )_{ j \in \N } \subset W^{1,p}( \Omega; \R^{n} ) $ that
\[ \| \nabla u_{j} \|_{ L^{p} } 
> j 
\bigg( \| \dist \big( I + \nabla u_{j} , SO(n) \big) \|_{ L^{p} } 
+ \| u_{j} \|_{ L^{p} } \bigg). \]
Since, for every $ X \in \M{n}{n} $, $| X | \le \dist \big( I + X , SO(n) \big) + 2 \sqrt{n}$, we have for any $ u \in W^{1,p}( \Omega ; \R^{n} ) $
\[ \| \nabla u \|_{ L^{p} } \le 
\| \dist \big( I + \nabla u , SO(n) \big) \|_{ L^{p} } 
+ 2 \sqrt{n} | \Omega |^{1/p} . \]
Applying the latter inequality to our sequence we obtain
\[ 2 \sqrt{n} | \Omega |^{1/p}
   > \big( 1 - \tfrac{1}{j} \big) \| \nabla u_{j} \|_{ L^{p} } + \| u_{j} \|_{ L^{p} }. \]
Hence, $ ( \nabla u_{j} )_{ j \in \N } $ is bounded in $ L^{p}( \Omega; \M{n}{n} ) $. From our assumption it follows immediately that $ u_{j} \to 0 $ in $ L^{p}( \Omega; \R^{n} ) $ and consequently $u_{j} \weakly 0$ in $W^{1,p}( \Omega; \R^{n} )$. Furthermore, by Theorem \ref{theo:geometric-rigidity} there exists a sequence $ ( Q_{j} )_{ j \in \N } \subset SO(n) $, such that
\[ \| I + \nabla u_{j} - Q_{j} \|_{ L^{p} }
\le C \ \| \dist \big( I + \nabla u_{j} , SO(n) \big) \|_{ L^{p} } 
\le \tfrac{ C }{j} \ \| \nabla u_{j} \|_{ L^{p} }. \]
Hence $ I + \nabla u_{j} - Q_{j} \to 0 $ in $ L^{p}( \Omega; \M{n}{n} ) $. This, together with
$ \nabla u_{j} \weakly 0 $ in $ L^{p}( \Omega; \M{n}{n} ) $, yields $ I - Q_{j} \weakly 0 $ in $ L^{p}( \Omega; \M{n}{n} ) $.
But the latter is a sequence of constant matrices and therefore converges strongly. Therefore,
\[ u_{j} \to 0 \quad \mbox{in} \quad W^{1, p}( \Omega ; \M{n}{n} ). \]
Let us, if necessary, pass to a not-relabelled subsequence so that $\| \nabla u_{j} \|_{ L^{p} } \le \frac{1}{j}$. Since $\| I - Q_{j} \|_{L^p} \le 2C \| \nabla u_{j} \|_{ L^{p} }$, we can, at least from some $j_{0}$ on, estimate
\[ | I - ( Q_{j} )_{\sm} | \le C' \| \nabla u_{j} \|_{ L^{p} }^{2} \le \frac{ C' }{j} \| \nabla u_{j} \|_{ L^{p} }. \]
Hence,
\begin{align*} 
\frac{ C }{j} \| \nabla u_{j} \|_{ L^{p} } 
  & \ge \| I + \nabla u_{j} - Q_{j} \|_{ L^{p} } \\ 
  & \ge \| I + e(u_{j}) - ( Q_{j} )_{ \sm } \|_{ L^{p} } 
  \ge \| e(u_{j}) \|_{ L^{p} } - \frac{ C' | \Omega |^{1/p} }{j} \| \nabla u_{j} \|_{ L^{p} }.
\end{align*} 
It follows
\[ \frac{ C' | \Omega |^{1/p} + C + 1 }{j} \| \nabla u_{j} \|_{ L^{p} } 
> \| e(u)_{j} \|_{ L^{p} } + \| u_{j} \|_{ L^{p} }, \]
which contradicts Korn's inequality (cf.\ Theorem \ref{theo:Korn}). 
\end{proof}

\subsubsection*{A version of Zhang's theorem}

When remarking on the non-degeneracy of the homogenized energy functionals we made use of the following extension of a theorem of Zhang (see \cite{Zhang}) to general $p$. 

\begin{theo}\label{theo:p-Zhang}
For $ 1 < p < \infty $ it holds
\[ [ \dist^{p}( \cdot , SO(n) ) ]^{\qc} 
   \ge C_{\rm geo}^{-p} \, \dist^{p}( \cdot , SO(n) ) \]
where $ C_{\rm geo} $ is the constant from the geometric rigidity result for unit cube $(0,1)^n$ and for the exponent $p$.
\end{theo}

\begin{proof}
By Theorem \ref{theo:geometric-rigidity} for any $\varphi \in C^{\infty}_{c}((0,1)^n)$ there exists an $R \in {\rm SO}(n)$ such that
\begin{align*}
  &\int_{(0,1)^n} \dist^{p} \big( X + \nabla \varphi(x) , {\rm SO}(n) \big) \, dx 
   \ge C_{\rm geo}^{-p} \, \int_{(0,1)^n} | X + \nabla \varphi(x) - R |^{p} \ dx \\
  &\ge C_{\rm geo}^{-p} \, \bigg| \int_{(0,1)^n} \big( X + \nabla \varphi(x) - R \big) \ dx \bigg|^{p} 
   = C_{\rm geo}^{-p} | \, X - R |^{p} 
  \ge C_{\rm geo}^{-p} \, \dist^{p}( X , {\rm SO}(n) ), 
\end{align*} 
where in the second step we have applied Jensen's inequality. Consequently, 
\begin{align*}
  \big[ \dist^{p}( \cdot , {\rm SO}(n) ) \big]^{\qc}(X)
  &= \inf_{ \varphi \in C^{\i}_{c}((0,1)^n) } \int_{(0,1)^n} 
    \dist^{p} \big( X + \nabla \varphi(x) , {\rm SO}(n) \big) \, dx \\
  &\ge C_{\rm geo}^{-p} \, \dist^{p}( X , {\rm SO}(n) ). \qedhere
\end{align*} 
\end{proof}

\end{appendix}

 \typeout{References}

\end{document}